\documentclass[notitlepage, 11pt]{article}

\usepackage{color}
\usepackage{amsmath}
\usepackage[mathscr]{eucal}
\usepackage{latexsym,amssymb, upgreek}
\usepackage{graphicx}
\usepackage{mathrsfs}
\allowdisplaybreaks

\newtheorem{lemma}{Lemma}
\newtheorem{theorem}{Theorem}
\newtheorem{corollary}{Corollary}
\newtheorem{definition}{Definition}

\newtheorem{proposition}{Proposition}

\newtheorem{assumption}{Assumption}

\newtheorem{remark}{Remark}

\newcommand{\beginsec}{
\setcounter{lemma}{0}
\setcounter{theorem}{0}
\setcounter{corollary}{0}
\setcounter{definition}{0}
\setcounter{example}{0}
\setcounter{proposition}{0}
\setcounter{condition}{0}
\setcounter{assumption}{0}
\setcounter{conjecture}{0}
\setcounter{problem}{0}
\setcounter{remark}{0}
}

\newcommand{\noi}{\noindent}

\newcommand{\ra}{\rightarrow}
\newcommand{\Ra}{\Rightarrow}

\newcommand{\E}{\mathbb{E}}
\newcommand{\R}{\mathbb{R}}

\newcommand{\p}{\mathbb{P}}
\newcommand{\N}{\mathbb{N}}
\newcommand{\I}{\mathcal{I}}

\newcommand{\la}{\lambda}

\newcommand{\eps}{\varepsilon}

\newcommand{\ph}{\varphi}

\newcommand{\al}{\alpha}

\newcommand{\del}{\delta}
\newcommand{\om}{\omega}

\newcommand{\Gam}{\mathnormal{\Gamma}}
\newcommand{\Del}{\mathnormal{\Delta}}
\newcommand{\Th}{\mathnormal{\Theta}}

\newcommand{\Om}{\mathnormal{\Omega}}

\newcommand{\C}{{\mathbb C}}
\newcommand{\D}{{\mathbb D}}

\newcommand{\Z}{{\mathbb Z}}

\newcommand{\EE}{{\mathbb E}}

\newcommand{\calA}{{\cal A}}
\newcommand{\calB}{{\cal B}}

\newcommand{\calD}{{\cal D}}
\newcommand{\calE}{{\cal E}}
\newcommand{\calF}{{\cal F}}

\newcommand{\calI}{{\cal I}}

\newcommand{\calK}{{\cal K}}

\newcommand{\calM}{{\cal M}}

\newcommand{\calQ}{{\cal Q}}
\newcommand{\calR}{{\cal R}}
\newcommand{\calS}{{\cal S}}

\newcommand{\calY}{{\cal Y}}
\newcommand{\calZ}{{\cal Z}}

\newcommand{\bC}{{\mathbf C}}

\newcommand{\lan}{\langle}

\newcommand{\ran}{\rangle}

\newcommand{\skp}{\vspace{\baselineskip}}

\newcommand{\w}{\wedge}
\newcommand{\lt}{\left}
\newcommand{\rt}{\right}

\newcommand{\bvnorm}[1]{[\kern-0.45ex[\kern0.1ex #1 \kern0.1ex]\kern-0.45ex]}

\newcommand{\mean}[1]{\langle#1\rangle}
\newcommand{\To}{\Rightarrow}

\newcommand{\iy}{\infty}
\newcommand{\osc}{\text{osc}}

\newcommand{\B}{{\cal B}}

\newcommand{\up}{\uparrow}
\newcommand{\Dup}{\mathbb{D}_{\calM}^\uparrow}
\newcommand{\bal}{\boldsymbol{\alpha}}
\newcommand{\bbeta}{\boldsymbol{\beta}}
\newcommand{\bxi}{\boldsymbol{\xi}}
\newcommand{\biota}{\boldsymbol{\iota}}

\newcommand{\bmu}{\boldsymbol{\mu}}

\newcommand{\qed}{\hfill $\Box$}

\newcommand{\bm}{\mathbf{m}}
\newcommand{\ind}{\mathbf{1}}

\usepackage{latexsym}
\usepackage{amsmath, amssymb, amsfonts}

\numberwithin{equation}{section}

\title{Law of large numbers for the many-server earliest-deadline-first queue}
\author{Rami Atar\footnote{Department of Electrical Engineering, Technion--Israel Institute of Technology, Haifa 32000, Israel.}
\hspace{4em} Anup Biswas\footnote{Department of Mathematics, Indian Institute of Science Education and Research, Pune 411008, India.}
\hspace{4em} Haya Kaspi\footnote{Department of Industrial Engineering and Management,
Technion--Israel Institute of Technology,
Haifa 32000, Israel.}
 }

\date{May 3, 2017}

\begin{document}

\maketitle

\begin{abstract}
A many-server queue operating under the earliest deadline first discipline, where the distributions
of service time and deadline are generic, is studied at the law of large numbers scale.
Fluid model equations, formulated in terms of the many-server transport equation
and the recently introduced measure-valued Skorohod map,
are proposed as a means of characterizing the limit.
The main results are the uniqueness of solutions to these equations,
and the law of large numbers scale convergence to the solutions.

\skp

\noi{\bf AMS subject classifications:}\, 60F17, 60G57, 68M20

\skp

\noi{\bf Keywords:}\, measure-valued processes,
measure-valued Skorohod map, many-server transport equation,
fluid limits, earliest-deadline-first, least-patient-first, many-server queues
\end{abstract}

\section{Introduction}\label{intro}

This paper proves a law of large numbers (LLN) many-server limit
for a queueing model with general service time and deadline distributions,
operating under the {\it earliest-deadline-first} (EDF) scheduling policy
(we refer to this model as G/G/N+G EDF).
By a {\it many-server limit} we refer to a setting where the number of servers
grows without bound.
The limit is characterized in terms of a set of so-called {\it fluid model equations} (FME)
that involve both the {\it many-server transport equation} (MSTE) \cite{kaspi-ramanan}
and the recently introduced {\it measure-valued Skorohod map} (MVSM)
\cite{atar-bis-kaspi-ram}.
It provides the first result on the EDF policy involving a many-server limit.
Several papers have analyzed EDF asymptotically
by appealing to the so-called {\it frontier process} (see below).
However, as argued in \cite{atar-bis-kaspi-ram}, the method based on this process
is not generic enough to cover a large variety of models (especially ones with
time-varying parameters).
Our motivation is to extend the asymptotic analysis of EDF to settings where
the method involving frontier process is not expected to be effective; the many server regime
offers a natural setting of this sort (even when the parameters are constant over time).

In recent years the use of measure-valued processes in mathematical modeling of
queueing systems has been very successful.
As far as many-server asymptotics are concerned, it is well understood since
the work of Halfin and Whitt \cite{hal-whi} that exponential service time distribution
leads to simple limit dynamics; specifically, the diffusion-scale heavy traffic
limit of \cite{hal-whi} is characterized by a diffusion process on the real line.
However, there is a great deal of motivation to study many-server
models under general service time distributions, stemming from applications
such as call centers and cloud computing. For example,
the statistical study of a call center by Brown et al.\ \cite{brown}
finds a good fit of the service time data to the lognormal distribution.
In this vein, in \cite{whitt}, Whitt considered a G/G/N system with
abandonment and proposed a deterministic LLN (otherwise referred to as fluid) approximation.
In \cite{kaspi-ramanan}, Kaspi and Ramanan
obtained measure-space valued fluid limits for such systems.
Kang and Ramanan generalized this work by modeling customer abandonment
\cite{kang-ramanan}. Further generalizations in this direction
were obtained by Zhang \cite{zhang} and Walsh-Zu\~{n}iga \cite{zuniga}. In \cite{atar-kas-shim},
Atar et al.\ studied multi-class many-server queues with fixed priority and
established the existence of a unique fluid limit.
Kang and Ramanan studied ergodic properties of
the G/G/N+G model and its relation with the invariant states
of the fluid limit in \cite{kang-ramanan1}.
Reed \cite{reed} established the fluid and diffusion limits
of the customer-count processes for many-server queueing systems
under a finite first moment assumption on service time distribution.

The aforementioned works were concerned with the {\it first-come first-served} (FCFS)
discipline. Many-server systems operating under the EDF discipline were considered
recently by Mandelbaum and Mom\'{c}ilovi\'{c} \cite{mandel-mom}, where
a fluid limit heuristic was developed. Motivated by prioritizing customers with least patience
and emphasizing the importance of this policy for emergency services, \cite{mandel-mom}
refers to this policy as {\it least-patience-first}.
Decreusefond and Moyal \cite{Dec-Moyal} study the fluid limits of M/M/1+M EDF,
and Atar et al.\ \cite{atar-bis-kaspi} generalize these results to G/G/1+G EDF.
Diffusion limits for G/G/1+G EDF systems
were studied by Doytchinov et al.\ \cite{Doy-Leh-Shre}
and Kruk et al.\ \cite{kruk-lehoc-ram-shre}.
For additional work on EDF in asymptotic regimes other than the many-server limit
we refer to Kruk \cite{kruk} and references therein.

An attractive feature of EDF, established in several of the aforementioned settings,
is that it minimizes the abandonment count.
Specifically, in \cite{Moyal, panwar-towsley, panwar-towsley-wolf} it was
shown for a single server model that EDF minimizes
customer abandonments within a certain class of scheduling policies.
The paper \cite{kruk-lehoc-ram-shre} studies G/G/1+G and shows
that the {\it reneged work} is minimized under EDF.
Comments in the introduction of \cite{mandel-mom} also address this minimality property,
and so do some of the results of Section 4.1.5 of \cite{atar-bis-kaspi-ram}.

In this article we are interested in the many-server LLN limit
of the G/G/N+G EDF. To elaborate on the hurdles
in obtaining fluid limits in this setting
let us briefly mention some tools that have been used in the literature
to treat the LLN for the single server EDF.
The aforementioned frontier process has been one of the main tools in
\cite{atar-bis-kaspi, Doy-Leh-Shre, kruk-lehoc-ram-shre}.
One defines the lead time of a customer at time $t$ as the (possibly negative)
difference between the customer's deadline and the time $t$.
The frontier process at time $t$ is defined as the maximum lead time at $t$
of all the customers that have ever been in service in the interval $[0, t]$.
The method developed in the papers mentioned above relies on the validity of the property
that, in an asymptotic sense, the frontier at any given time separates the population of
customers to those that have been sent to service and those that are still in the buffer,
according as their relative deadline at that time is below the value of the frontier
or above it, respectively. The main idea of using the frontier process asymptotics
to characterize the asymptotics of the full model is that when this property is valid,
the frontier can be recovered directly from the model's
primitive data (specifically, in terms of a one-dimensional Skorohod map),
and at the same time the full state of the system, including the measure-valued queueing process,
can be expressed in terms of the frontier.

It was argued in \cite{atar-bis-kaspi-ram}
that this method breaks down for more general settings than those considered in
the above papers, because the property by which the two
populations are asymptotically separated by the frontier, fails to hold.
Such settings include the single-server model at fluid scale
with time varying characteristics (such as rate of arrivals, service
or patience distribution). An alternative approach developed in
\cite{atar-bis-kaspi-ram} was to introduce a certain Skorohod-type transformation
in the space of paths taking values in the space of finite measures on the real line,
referred to as a measure-valued Skorohod map (MVSM) (see Proposition~\ref{prop2.1}).
It was used there to obtain fluid limits of several queueing models,
including the G/G/1+G EDF with (possibly) time varying characteristics.

Consider now the many-server asymptotics of the G/G/N+G system.
As far as the FCFS discipline is concerned,
it is well understood, and captured by the MSTE of \cite{kaspi-ramanan},
that the time evolution of the ages of jobs in the service pool
affects the rate at which jobs are transferred from the buffer to the pool.
In particular, this rate varies over time even for a model
with time homogeneous characteristics, unless the system is at equilibrium
(or when the service time distribution is exponential).
Now, the time evolution of the content of the buffer under
variable rate of transferring jobs to the service pool is much like that
of a single server model for which the service rate varies over time.
In light of this analogy and the discussion above regarding the expected failure
of the frontier process method for systems with time varying characteristics,
one does not expect this method to be useful
for analyzing a {\it time homogeneous} G/G/N+G EDF system off equilibrium.
The approach proposed in this paper is to appeal to the MVSM instead.

The aforementioned MSTE and MVSM serve in this work as two main building blocks.
The former is used to describe the evolution of the collection of age processes,
that keeps track of the time jobs spend in service. The latter captures the EDF
discipline, and allows one to express the reneging count process (see Theorem~\ref{Thm2}).
The representation of the FME is provided in terms of these two building blocks.
The first main difficulty we address is the uniqueness of solutions to the FME
(see Section~\ref{fluid eqn}). Uniqueness cannot be obtained from
\cite{kang-ramanan} due to the quite different structure of the FME,
as well as to the fact that a certain monotonicity property
of the reneging count process with respect to the queue length, used crucially
in \cite{kang-ramanan} (see especially (3.14) there), is not clear in our setting.
Further, in \cite{atar-bis-kaspi-ram}, uniqueness of the FME is established
via a minimality result (see Theorem~3.1 there),
which holds when the service rate is strictly positive.
In a many-server setting the rate at which jobs are transferred from the buffer
to the service pool might get arbitrarily close to zero even when the system
is busy, and thus the arguments from \cite{atar-bis-kaspi-ram} do not seem to apply.
Moreover, unlike the single server EDF model, treated in \cite{atar-bis-kaspi-ram},
where the service rate was part of the data, in our setting the rate of job departure
from the system (and as a result also the transfer rate alluded to above)
is part of the solution, which makes the uniqueness proof significantly more difficult.
A second hurdle is the convergence, where several technicalities must be addressed,
especially the identification of subsequential limits as solutions of the FME.
Our main two results are the uniqueness of FME solutions (Theorem~\ref{Thm1})
and the LLN scale convergence (Theorem \ref{Thm3}). They are both established
under certain assumptions, namely Assumption \ref{Assump3} and \ref{Assump4},
that are treated separately.

As a byproduct of the two main results
we obtain a LLN limit for G/G/N+D FCFS (see Corollary~\ref{Cor2}),
where `D' stands for deterministic deadline (this case was not covered by \cite{kang-ramanan, zhang, zuniga}).
Let us remark that we could also study scaling limits of other performance measures associated to G/G/N+G EDF, for instance,
waiting time. The analysis of waiting time in our settings would be very similar to that available in \cite{kaspi-ramanan, kang-ramanan}, and
therefore we do not pursue  it here.

The outline of this paper is as follows. At the end of this section we provide our basic notation.
Section \ref{sec-description} describes the model and its scaling, introduces the MVSM
and the MSTE and uses them to represent the model dynamics.
Section \ref{main-result} formulates the FME
and states the two main results: the uniqueness of solutions of these equations,
and the convergence of the scaled processes to their solution.
Section \ref{sec-fme} provides the proof of uniqueness of solutions.
Section \ref{sec-tight} establishes tightness of the scaled processes, and finally,
Section \ref{sec-char} completes the proof of convergence.

\skp

\noindent{\bf Notation.}\,
The set of nonnegative real numbers is denoted by $\R_+$. For $x, y\in\R$, $x\vee y=\max(x,y)$,
$x\wedge y=\min(x,y)$, $x^+=x\vee 0$ and $x^-=(-x)\vee 0$.
For $A\subset \R$, $\ind_A$ denotes the indicator function of $A$. Given a metric space
$\calS$, $\C_b(\calS)$ and $\C_c(\calS)$ are, respectively, the space of real-valued
bounded continuous functions and the space of real-valued continuous functions
with compact support defined on $\calS$.
When $\calS$ is a subset of a finite dimensional vector space, we denote the set of continuously differentiable functions on $\calS$ with compact support by $\C^1_c(\calS)$.
Let $\C_{b, +}(\calS)$ denote the subset of $\C_b(\calS)$ of $\R_+$-valued functions.
Let $\D_\calS(\R_+)$ and $\C_{\calS}(\R_+)$, abbreviated as $\D_\calS$ and $\C_\calS$,
denote the spaces of functions
$\R_+\to \calS$ that are right continuous with finite left limits (RCLL),
and respectively, continuous.
It is always assumed that $\D_\calS$ is endowed with the $J_1$ topology \cite{ethier-kurtz}.
Denote by $\D_{\R_+}^\up$ (resp., $\C_{\R_+}^\up$)
the subset of $\D_{\R_+}$ (resp., $\C_{\R_+}$) of nondecreasing functions.
For $\varphi:\R_+\to\R$, define
$\|\ph\|_T=\sup_{s\in [0, T]}|\ph(s)|$ for $T<\infty$,
$\|\ph\|_\infty=\sup_{s\in[0, \infty)}|\ph(s)|$,
and given $\del>0$,
$$
\osc_\del(\varphi, T)=\sup\{|\varphi(s)-\varphi(t)|\, : \, |s-t|\leq\del, \; s, t\in [0, T]\}.
$$
For functions $\ph=\varphi(x, t)$ defined on $\R^n\times \R$ we write
$\varphi_x$ and $\varphi_t$ for the partial derivatives with respect to the first ($x\in\R^n$)
and second ($t\in\R$) variable, respectively.

The space of non-negative finite Borel measures on a Polish space $\calS$
is denoted by $\calM(\calS)$
and the Borel $\sigma$-field of $\calS$ is denoted by $\calB(\calS)$.
Given $a<\infty$, $\calM_a(\calS)$ denotes the subset of $\calM(\calS)$
consisting of measures with total mass less or equal to $a$.
$\calM^0(\calS)$ denotes the class of atomless measures.
We abbreviate $\calM(\R_+)$ by $\calM$ and $\calM^0(\R_+)$ by $\calM^0$.
For  any $\mu\in\calM(\calS)$ and
Borel measurable function $g$ on $\calS$, denote $\lan g, \mu\ran= \int g d\mu$. Endow $\calM(\calS)$ with the Prohorov metric, denoted by $d_\calM$.
It is well known that $(\calM(\calS), d_\calM)$ is a Polish space \cite[Appendix]{daley-verejones},
and that this topology on $\calM(\calS)$ is equivalent to the weak topology on this space.
Denote by $\del_x$ the unit mass at the point $x$.
Denote convergence in distribution by the symbol `$\To$'.
Finally, for  $\zeta\in\D_{\R_+}^\up$, denote by $\bm^\zeta$ the Lebesgue-Stieltjes measure that
$\zeta$ induces on $(\R_+, {\mathcal B}(\R_+))$, namely,
\[
\bm^\zeta(B)=\zeta(0)\del_0(B)+\int_{(0,\iy)}\ind_B(t)d\zeta_t,  \quad
B \in {\mathcal B}(\R_+).
\]
Throughout, we write ``$d\zeta$-a.e.''  to  mean ``$d\bm^\zeta$-a.e.''

\section{Model and state dynamics}\label{sec-description}
\beginsec

\subsection{Model description}\label{sec21}

Consider a sequence of systems indexed by $N\in\N$, where the  $N$-th system
has $N$ servers that work in parallel.
Each system also has a buffer of infinite capacity, in which arriving customers
are served according to a non-idling, non-preemptive EDF policy described as follows.
Customers arrive with specified deadlines.
An arriving customer enters service immediately if there is a free server at the time of its arrival.
On the event that all servers are busy, it is queued. When a server becomes available
(and the queue is nonempty), it picks the customer that has the earliest deadline
among those that are in the queue. Ties are broken according to the order of arrival.
Customers that do not enter service by the time
of their deadline leave the system (however, customers do not renege while being served).
All servers are identical and capable of serving all the customers.
The stochastic processes associated with the model, to be
introduced below, are all defined on a common probability space $(\Om, \calF, \p)$.
The symbol $\E$ denotes the expectation with respect to $\p$.

A process with right-continuous nondecreasing $\Z_+$-valued sample paths
is referred to as
a {\it counting process}. A counting process that has only jumps of size 1 and starts
at zero is said to be {\it simple}.
Let $E^N$ be a simple counting process that accounts for arrivals in the $N$-th system.
Namely, the number of arrivals in the time interval $[0, t]$ is given by $E^N_t$.
The jump times of this process, that we denote by $\{a_i^N,\; i\in\N\}$,
correspond to the arrival times. This sequence is nondecreasing, so that
$a^N_i$ gives the arrival time of the $i$-th customer.
The number of customers in the system (i.e., in service or in the queue)
at time $0$ is denoted by $X^N_0$.
The sequence $\{a_i^N\}$ is next extended to $\calI^N:=\{-X^N_0+1,-X^N_0+2,\ldots,0\}\cup\N$, so that
$\{a^N_i, i\leq 0\}$ give the arrival times of the $X^N_0$ customers present at time $t=0$.
Note that $\calI^N$ is a random set of indices.

To model deadlines,
let $\{r^N_j, j\in \Z\}$ be an i.i.d.\ sequence of positive random variables.
For $i\in\calI^N$, $r^N_i$ represents the {\it patience time} of the customer $i$, that is,
the deadline of the customer relative to its arrival time. Thus the deadline of customer $i$
is given by $u^N_i=a^N_i + r^N_i$, $i\in \I^N$.
We will always refer to $u^N_i$ as the \textit{absolute deadline} of the
$i$-th customer, to avoid confusion with what is referred to in \cite{atar-bis-kaspi-ram}
as {\it relative deadline} (which indicates the deadline with respect to the current time,
and is elsewhere referred to
as the {\it remaining patience time} or \textit{lead time}
\cite{atar-bis-kaspi, Doy-Leh-Shre, kruk-lehoc-ram-shre}).
To recapitulate the reneging rule, customer $i\ge1$ that does not start
service by time $u^N_i$ leaves the system at that time.

Let $\calQ^N_0\in \calM$ be a measure describing the state of the queue at time $0$,
such that for $B\in\calB(\R_+)$, $\calQ^N_0(B)$ represents
the number of those customers in the queue at time $0$ whose absolute deadlines are in $B$.
Define a measure-valued process $\calE^N$, with sample paths in $\D_{\calM}$,
as follows. For $B\in\calB(\R_+)$ and $t\geq 0$,
\begin{equation}\label{00}
\calE^N_t(B)= \sum_{i=1}^{E^N_t}\del_{u^N_i}(B)=\sum_{i=1}^{E^N_t}\ind_B(u^N_i).
\end{equation}
Then $\calE_t(B)$ gives the number of arrivals in $[0, t]$ whose absolute deadlines lie in $B$.
Define $\al^N$, a process with sample paths in $\D_{\calM}$, as
\begin{equation}\label{8}
\al^N_t =\calQ^N_0 + \calE^N_t.
\end{equation}
Introduce the class
\begin{equation}\label{29}
\Dup =\Big\{\zeta\in \D_{\calM}\, :\, \text{the map
$t\mapsto\mean{f, \zeta_t}$ is in $\D_{\R_+}^\up$ for every}\, f\in\C_{b, +}(\R_+)\Big\}.
\end{equation}
Since by definition, $t\mapsto\calE^N_t(B)$ are nondecreasing,
it is easy to see that the sample paths of $\al^N$ lie in $\Dup$.
It is shown in \cite[Lemma~2.1]{atar-bis-kaspi-ram} that $\Dup$ forms a closed subset of $\D_{\calM}$. The MVSM, to be introduced in the next subsection,
is defined on this space.

To model service times, consider an i.i.d.\ sequence, $\{v_i,\, i\in\Z\}$,
with common cumulative distribution function $G$ on $[0, \iy)$. For $i\in\calI^N$,
the service time of customer $i$ is $v_i$.
We assume that $G$ has density, and denote it by $g$. Let
$$
H^s=\sup\{x\in[0, \iy)\; :\; G(x)< 1\}
$$
be the right end of the support of $g$. This constant may be finite or $+\iy$.

We assume that, for each $N$,
\begin{itemize}
\item the arrival process $E^N$, the sequence of service requirements $\{v_i,\, i\in\Z\}$ and the sequence of patience times $\{r^N_i,\,  i\in\Z\}$ are mutually independent.
\end{itemize}

We refer to the time spent by a customer in service as its {\it age in service},
or simply its {\it age} \cite{atar-kas-shim, kang-ramanan, kaspi-ramanan}.
For a customer $i\in\calI^N$ that is ever admitted into service, let $\gamma^N_i$
denote the admittance time. Let $\gamma^N_i=\iy$ if this customer reneges.
Thus, for a customer $i$ initially in service, $\gamma^N_i\le0$.
We emphasize that $\gamma_i^N$ is associated with the $i$-th customer to arrive into the system
rather than the $i$-th customer to be admitted.
The age process $\om^N_i$ associated with customer $i\in\calI^N$ is defined as
\[
\om^N_i(t)=(t-\gamma^N_i)^+\w v_i,\qquad t\ge0.
\]
Note that the age of a customer is zero at time $t$ if it has not entered service by that time,
and that the age process is identically zero for those customers that renege.
Let $K^N$ be a counting process representing cumulative number
of admittances into service since time 0 (specifically, $K^N_0=0$).
Since the number of customers initially in the system is given by $X^N_0$
and $N$ is the number of servers, the number of customers initially
in service is given by $X^N_0\w N$.
For $t\in[0, \iy)$, let $\nu^N_t$ be the discrete measure on $[0, H^s)$
recording the collection of ages, given by
\begin{equation}\label{1}
\nu^N_t=\sum_{i=-X^N_0+ 1}^{E^N_t}\del_{\om^N_i(t)}\ind_{\{\om^N_i(t)<v_i,\ t\ge\gamma^N_i\}}.
\end{equation}
Note that $\lan\nu^N_t,1\ran=X^N_0\w N$.
By construction, the measure-valued process $\nu^N_t$ has sample paths in
$\D_{\calM([0, H^s))}$.

Let $D^N$ be a counting process representing
cumulative number of departures from service. We have the explicit representation
\begin{equation}\label{2}
D^N_t=\sum_{i=-X^N_0 + 1}^{E^N_t}\sum_{s\in[0, t]}
\ind_{\{\frac{d\om^N_i}{dt}(s-)>0, \frac{d\om^N_i}{dt}(s+)=0\}}.
\end{equation}
Let $R^N$ be a counting process representing cumulative number of reneging customers.
Let $X^N$ and $Q^N$ be $\Z_+$-valued processes
representing the number of customers in the system and in the queue, respectively.
Note that the number of customers in service is given by $\lan 1,\nu^N\ran$,
hence $X^N=Q^N+\lan 1,\nu^N\ran$.

The balance equations for the content of the queue,
the service station and the system, respectively, are as follows
\begin{align}
Q^N_0 + E^N_t &=  Q^N_t + K^N_t+R^N_t,\quad\;t\geq 0,\label{3}
\\
X^N_t &=Q^N_t+\lan 1,\nu^N_t\ran\,, \label{AB1}
\\
\label{7}
\mean{1, \nu^N_0}+K^N_t &= \mean{1, \nu^N_t} + D^N_t,\quad\; t\geq 0,
\\
X^N_0 + E^N_t & = X^N_t + D^N_t + R^N_t, \quad\; t\geq 0.\label{4}
\end{align}
Since the system is working under a nonidling policy, we also have
\begin{align}\label{5}
Q^N_t = [X^N_t-N]^+, \quad \text{and}\; \mean{1, \nu^N_t}= X^N_t\wedge N,\qquad t\ge0.
\end{align}
It is also evident from our description that
reneging can only occur when all the servers are busy, and so
\begin{equation}\label{6}
\int_0^\cdot [N-X^N_t]^+dR^N_t=0.
\end{equation}

We next introduce the filtration that captures the information available as a function of time.
Following \cite{kang-ramanan, kaspi-ramanan}
we introduce a {\it station process} $s^N=\{s^N_i, i\in\calI^N\}$.
Assume that the individual servers are labeled by $\{1,\ldots,N\}$.
For $t\ge0$, if customer $i$ has entered service by time $t$ then $s^N_i(t)$
is the number of the server at which it has started (and, possibly, completed)
service. Otherwise, $s^N_i(t)=0$.
For $t\in[0, \iy)$, let $\tilde\calF^N_t$ be the $\sigma$-algebra generated by
$$
\{X^N_0, \calQ^N_0, \calE^N_s, \om^N_i(s), s^N_i(s),\; i\in \calI^N,\; s\in[0,t]\},
$$
and let $\{\calF^N_t\}$ denote the associated right-continuous filtration, complete with respect to $\p$.

\subsection{Representation in terms of the MSTE and the MVSM}\label{meas-char}

We introduce the MVSM and then use it to describe the dynamics of the measure-valued process
defined above.
To this end, we first introduce the {\it measure-valued Skorohod problem}
(MVSP) from \cite{atar-bis-kaspi-ram}.

\begin{definition}[{\bf MVSP}]\label{D-mvsm}
Let $(\bal, \bmu)\in \D^\uparrow_{\calM}\times\D^\uparrow_{\R_+}$. Then
$(\bxi, \bbeta, \biota)\in\D_{\calM}\times\D^\uparrow_{\calM}\times\D^\uparrow_{\R_+}$ is said to solve the MVSP for the data $(\bal, \bmu)$ if, for each $x\in[0, \iy)$,
\begin{enumerate}
\item[1.] $\bxi_t[0, x] = \bal_t[0, x]-\bmu_t +\bbeta_t(x, \iy) + \biota_t,
 \text{ for all }\; t\geq 0$,
\item[2.] $\bxi_t[0, x] = 0,\; d\bbeta_t(x, \iy)$-a.e.
\item[3.] $\bxi_t[0, \iy) = 0,\;  d\biota_t$-a.e.
\item[4.] $\bbeta_t[0, \iy)+\biota_t=\bmu_t,
 \text{ for all }\; t\geq 0$.
\end{enumerate}
\end{definition}
Recall the definition \eqref{29}, and define $\C_{\calM}^\up$ analogously
as a subset of $\C_{\calM}$. Let
$\C^{\uparrow, 0}_ {\calM}$ denote the collection of members $\zeta\in\C_{\calM}^\up$
for which the measure $\zeta_t$ is atomless for each $t$.
The following was proved in \cite[Proposition 2.8 and 2.10]{atar-bis-kaspi-ram}.
\begin{proposition}\label{prop2.1}
\begin{enumerate}
\item[1.] For every $(\bal, \bmu)\in\D^\uparrow_{\calM}\times\D^\uparrow_{\R_+}$
there exists a unique
$(\bxi, \bbeta, \biota)\in\D_{\calM}\times\D^\uparrow_{\calM}\times\D^\uparrow_{\R_+}$
that constitutes a solution to the MVSP with data $(\bal,\bmu)$.
\item[2.] Let $\varTheta$ denote the corresponding solution map, so that
$(\bxi, \bbeta, \biota)=\Th(\bal,\bmu)$.
Then $\varTheta$ is measurable. Moreover, $\varTheta$ is continuous
on $\C^{\uparrow, 0}_ {\calM}\times\C^\uparrow_ {\R_+}$.
\end{enumerate}
\end{proposition}
In addition to $\Th$ defined above we shall use the notation $\Th_1$, $\Th_2$ and $\Th_3$
for denoting the relations
$\bxi=\Th_1(\bal, \bmu), \, \bbeta=\Th_2(\bal, \bmu)$ and $\biota=\Th_3(\bal, \bmu)$.

\begin{remark}\label{R2.1}\rm
Observe that Definition~\ref{D-mvsm} bears close resemblance with the one dimensional Skorohod map.
That is, for a given $\psi\in\D_{\R}$, a pair $(\varphi, \eta)\in\D_{\R_+}\times\D^{\uparrow}_{\R_+}$
is said to solve the one-dimensional Skorohod problem for $\psi$,
if $\varphi=\psi+\eta$ and $\varphi(t)=0$ $d\eta$-a.e.
It is well known that there exists a unique pair
$(\varphi, \eta)\in\D_{\R+}\times\D^{\uparrow}_{\R_+}$ solving this problem for
a given $\psi\in\D_{\R}$. Moreover, if we denote
$\Gam[\psi]=(\varphi, \eta)$, $\Gam_1[\psi]=\varphi, \, \Gam_2[\psi]=\eta$, this solution
is given by
$$
\varphi(t)=\Gam_1[\psi](t)=\psi(t)-\inf_{s\in[0, t]}(\psi(s)\wedge 0),
\quad \text{and}\quad \eta(t)=\Gam_2[\psi](t)=\varphi(t)-\psi(t).
$$
The relation between the two maps $\Th$ and $\Gam$ can be made explicit
\cite[Lemma~2.7]{atar-bis-kaspi-ram}, as follows.
If $\Th(\alpha, \mu)=(\xi, \beta, \iota)$ then
\begin{equation}\label{theta-gamma}
(\xi[0, x], \beta(x, \infty)+\iota)= \Gam[\alpha[0, x]-\mu], \quad \text{for all}\; x\in\R_+.
\end{equation}
\end{remark}

Next, for any measurable function $\varphi$ on $[0, H^s)\times\R_+$
we define a process $D^N(\varphi)$, which takes values in $\R$, by
\begin{equation}\label{9}
D^N_t(\varphi) = \sum_{i=-X^N_0+1}^{E^N_t}\, \sum_{s\in[0, t]} \ind_{\{\frac{d\om^N_i}{dt}(s-)>0, \frac{d\om^N_i}{dt}(s+)=0\}} \varphi(\om^N_i(s),s).
\end{equation}
From the right continuity of $\{\calF^N_t\}$ it follows that
$D^N(\varphi)$ is $\{\calF^N\}$-adapted. Also, from \eqref{2} and \eqref{9} we see that
$D^N(1)=D^N$.
Three additional measure-valued processes that will be used in our analysis are
as follows. Let $\calQ^N$ be a process having sample paths in $\D_{\calM}$,
such that for each $t\in[0, \iy)$,
$\calQ^N_t(B), B\in\calB(\R_+)$, represents the total number of customers in the queue
at time $t$ whose absolute deadline is in $B$. This can be written as
$$
\calQ^N_t(B) = \sum_{i=-X^N_0+1}^{E^N_t}\ind_{\{t<\gamma^N_i\w u^N_i\}}\ind_{B}(u^N_i).
$$
For $B\in\calB(\R_+)$, let
$$
\calK^N_t(B)=\sum_{i=-X^N_0+1}^{E^N_t} \ind_{\{0\le\gamma^N_i\le t\}}\ind_B(u^N_i)
$$
denote the number of customers admitted into service by time $t$
whose absolute deadline is in $B$. Note the relation $\calK^N_t([0,\iy))=K^N_t$.
It is easy to see that $\calK^N$ has sample paths in $\D^\uparrow_{\calM}$.
Similarly, define a process $\calR^N$ with sample paths in $\D^\uparrow_{\calM}$ by
$$
\calR^N_t(B)=\sum_{i=-X^N_0+1}^{E^N_t} \ind_{\{t\ge u_i^N,\gamma^N_i=\iy\}}\ind_{B}(u^N_i).
$$
Thus $\calR^N_t(B)$ denotes the number of customers that reneged by time $t$,
whose absolute deadlines lie in $B$. Observe the relation
$$
\calR^N_t[0, x]=R^N_{x\wedge t}, \quad \text{for all }\, t,\, x\geq 0.
$$
Let
$$
\sigma^N_t=\inf\{x\in[0, \iy)\, :\, \calQ^N_t[0, x]>0\}
$$
denote the left end of the support of $\calQ^N_t$ (defined as $+\iy$ when $\calQ^N_t=0$).
Then the fact that customers with deadline within $[0,t]$ cannot be present in the queue
at time $t$ is expressed as
\begin{equation}
\calQ^N_t[0, t] =0, \quad \text{for all }\, t\geq 0,\label{10}
\end{equation}
and the fact that reneging does not occur prior to the time of the absolute deadline implies
\begin{equation}
\int_0^\cdot \ind_{\{\sigma^N_{t-}>t\}}dR^N_t =0.\label{11}
\end{equation}
Moreover, the prioritization according to absolute deadlines can be expressed by
\begin{equation}
\int_0^\cdot \calQ^N_t[0, x]\, d\calK^N_t(x, \iy)=0,\qquad x\ge0,
\label{15}
\end{equation}
and it is also obvious that
\begin{equation}\label{41}
\int_0^\cdot \calQ^N_t[0, x]\, d\calR^N_t(x, \iy)=0,\qquad x\ge0,
\end{equation}
as if there is a customer in the queue with absolute deadline smaller than $x$ at time $t$,
no customer with higher absolute deadline may renege at time $t$.

The following result identifies crucial relations satisfied by the model,
in terms of the MSTE, \eqref{12} below, and the MVSM, $\Th$.
\begin{theorem}\label{Thm2}
Given $\varphi\in\C^1_c([0, H^s)\times\R_+)$,
the processes $\nu^N,D^N(\ph), K^N, \calQ^N, \calK^N, \calR^N, \al^N$ and $R^N$ satisfy
\begin{equation}\label{12}
\mean{\varphi(\cdot, t), \nu^N_t}
= \mean{\varphi(\cdot, 0), \nu^N_0}
+ \int_0^t\mean{\varphi_x(\cdot, s)+\varphi_s(\cdot, s), \nu^N_s} ds
-D^N_t(\varphi) + \int_{[0, t]}\varphi(0, s)dK^N_s,
\end{equation}
and
\begin{equation}\label{13}
(\calQ^N, \calK^N+\calR^N, 0) = \varTheta(\al^N, K^N+R^N).
\end{equation}
\end{theorem}

\begin{proof}
Equation \eqref{12} is shown in \cite[Theorem 5.1]{kaspi-ramanan}.
As for \eqref{13}, it follows from the balance equations \eqref{3}--\eqref{4}
that, for $x\in[0, \infty)$,
\begin{equation}
\calQ^N_t[0, x] = \al^N_t[0, x] - \calK^N_t[0, x]-\calR^N_t[0, x]
= \al^N_t[0, x] - K^N_t-R^N_t + \calK^N(x, \iy)+\calR^N(x, \iy).\label{14}
\end{equation}
Defining $\beta^N=\calK^N+\calR^N$ and using relations \eqref{15}, \eqref{41} and \eqref{14}
shows that $(\calQ^N, \beta^N, 0)$ satisfies all elements of the definition of the MVSP
with respect to the data $(\al^N, K^N+R^N)$. As a result, \eqref{13} holds.
\qed
\end{proof}

As already mentioned, equation \eqref{12},
that gives the dynamics of the measure-valued process $\nu^N$,
originates from \cite{kaspi-ramanan}. The first term on its RHS is simply
an initial condition. The second term
accounts for the fact that both age and time increase at rate 1. The remaining
two terms correspond to the departure and, respectively, admittance-into-service process.

LLN-scaled version of the processes introduced in this section
is attained by downscaling each of them by a factor $N$.
The resulting processes are denoted with a bar, as in
$\bar\calQ^N = N^{-1}\calQ^N$, for $N\in\N$.
Specifically, $\bar U^N =N^{-1}U^N$ for each of the processes
$U^N= X^N, Q^N, K^N, R^N, D^N,\calE^N$, $\al^N,\nu^N,\calR^N,\calK^N$.

\section{Main results}\label{main-result}
\beginsec

\subsection{Assumptions}
We introduce the main assumptions. The first two will be in force throughout this article.
Define
\begin{align*}
\calS_0&=\Big\{(\calQ_0, \nu_0, \al, x)\in \calM^0\times\calM_{1}([0, H^s))
\times\C^{\uparrow, 0}_{\calM}\times\R_+\, :\\
&\hspace{9em}\, 1-\mean{1, \nu_0}=[1-x]^+,\,
\mean{1, \nu_0}+ \calQ_0(\R_+)= x\Big\}.
\end{align*}
The assumption regarding convergence of the arrival process is as follows.
\begin{assumption}\label{Assump1}
There exists  $(\calQ_0, \nu_0, \al, X_0)\in\calS_0$ with
\begin{equation}\label{16}
\al_t[0, x] = \calQ_0[0, x] + \int_0^t\ind_{\{x\geq s\}} \la_s \pi[0, x-s] ds, \quad \forall\; t,\; x\in [0, \iy),
\end{equation}
where $\la:\R_+\to\R_+$ is locally bounded and $\pi$ is a probability measure
with $\pi\{0\}=0$, such that, as $N\to\iy$,
$\bar\al^N\Ra \al$ in $\D_{\calM}$,
$\bar{X}^N_0\Ra X_0$ in $\R_+$, and $\bar{\nu}^N_0\Ra \nu_0$ in $\calM_1([0,H^s))$.
Moreover, for any $T\in(0,\iy)$, $\sup_N\EE\bar\al^N_T[0,\iy)<\iy$.
\end{assumption}
The assumption that the limiting measure-valued process
$\al_t$ takes the form \eqref{16} corresponds to assuming
an asymptotic rate of arrival that follows the function $\la_t$ and an asymptotic
patience that is distributed according to the measure $\pi$.
In particular, the rate of arrivals is allowed to vary with time, but the patience
distribution is assumed to be fixed. Working with fixed patience distribution
allows us to keep things simple as far as this aspect of the model is concerned.
However, extending the results to time-varying patience distribution
it is not a serious obstacle, and one could allow that under suitable assumptions.
Indeed, this has been done in the single-server setting in \cite{atar-bis-kaspi-ram};
see Assumption 4.5 in \cite{atar-bis-kaspi-ram} for a general structure of
patience, and Example 5.3 of \cite{atar-bis-kaspi-ram} for a discussion
of this matter.

As for the service time distribution,
the hazard rate function of $G$ is defined by
$$h(x)=\frac{g(x)}{1-G(x)}, \quad x\in [0, H^s).$$
It is easy to see that $h$ is locally integrable on $[0, H^s)$.

\begin{assumption}\label{Assump2}
There exists $L^s<H^s$ such that the function
$h$ is either bounded or lower-semicontinuous on $(L^s, H^s)$.
\end{assumption}

Our results require that either one of the following two assumptions hold.

\begin{assumption}\label{Assump3}
\begin{enumerate}
\item[1.] The density $g$ vanishes at most at finitely many points of $[0, H^s)$.
If $\calZ$ denotes the set of zeros of $g$ then $g$ is bounded away from zero on
any compact subset of $[0, H^s)\setminus\calZ$.

\item[2.] $\nu_0\in\calM^0$.
\end{enumerate}
\end{assumption}

\begin{assumption}\label{Assump4}
There exists $\kappa_1>0$ such that $\pi[0, \kappa_1]=0$ where $\pi$ is given in \eqref{16}.
\end{assumption}

Assumption~\ref{Assump2} was introduced in \cite{kaspi-ramanan} to guarantee  the convergence of sub-sequential limits  to a solution of the MSTE.
Assumption~\ref{Assump3} is satisfied by various distributions of practical interest,
such as
exponential, log-normal, Weibull, gamma, log-logistic etc.
When it is in force, we do not make any assumption on the distribution $\pi$ of patience time.
On the other hand, Assmption~\ref{Assump4} which imposes a condition on $\pi$,
allows us to treat an extended collection (e.g., uniform over an interval
of the form $[a,b]$ where $b>a>0$, Pareto) of service time distributions.
Let us mention that a similar condition to Assumption~\ref{Assump4}
was also used in \cite[expression (2.5)]{kruk-lehoc-ram-shre} to establish
diffusion limits for the G/G/1+G EDF.

\subsection{Fluid model equations and statement of results}\label{fluid eqn}

Recall from Remark~\ref{R2.1} that $\Gam=(\Gam_1, \Gam_2)$ is the 1-dimensional Skorohod map, that is, for any locally bounded
$\psi:\R_+\to\R$, $\Gam_1[\psi](t)=\psi(t)-\inf_{s\in[0,t]}[\psi(s)\w0]$,
and $\Gam_2[\psi](t)=-\inf_{s\in[0,t]}[\psi(s)\w0]$,
$t\ge0$.
\begin{definition}[{\bf FME}]
A quadruple $(\calQ, \nu, K, R)$ in
$\C_{\calM}\times \C_{\calM_1([0, H^s))}\times\C^\up_{\R_+}\times\C^\uparrow_{\R_+}$
is said to solve the \textit{FME} with given data
$(\calQ_0, \nu_0, \al, X_0)\in\calS_0$, for $\al$ as in \eqref{16}, if
the initial conditions of $\calQ$ and $\nu$ are consistent with the data $\calQ_0$ and $\nu_0$,
one has $\int_0^t\mean{h, \nu_s}\ ds <\iy$ for all $t\ge0$,
and, letting
\begin{equation}
  \label{1000}
  E_t=\al_t[0,\iy)-\al_0[0,\iy),
\end{equation}
\begin{equation}
  \label{1001}
  Q_t=\calQ_t[0,\iy),
\end{equation}
\begin{equation}
  X_t = Q_t + \lan 1, \nu_t\ran, \label{fme2}
\end{equation}
and
\begin{equation}
  \label{1002}
  D_t=\int_0^t \lan h, \nu_s\ran ds,
\end{equation}
the following relations are satisfied.
For $\varphi\in \C^{1}_c([0, H^s)\times [0, \infty))$,
\begin{align}\label{fme5}
\lan \varphi, \nu_t\ran & = \lan \varphi(\cdot, 0), \nu_0\ran + \int_0^t \lan \varphi_x(\cdot, s)+\varphi_s(\cdot, s), \nu_s\ran ds
- \int_0^t\lan h\varphi(\cdot, s), \nu_s\ran ds + \int_0^t \varphi(0, s)dK_s,
\end{align}
and
\begin{align}
Q_0 + E_t &= Q_t + K_t +R_t,\label{fme1}
\\
Q_t &=[X_t-1]^+, \label{fme3}
\\
K_t &= \lan1, \nu_t\ran-\lan 1, \nu_0\ran + D_t \label{fme4}
\\
\calQ & = \Th_1(\al, K+R), \label{fme6}
\\
\calQ_t[0, t)&=0, \quad \forall\, t\geq 0,\label{fme7}
\\
\int_0^\cdot[1-X_s]^+dR_s & = 0, \label{fme8}
\\
\sigma_t&=t \quad dR\text{-a.e., where for } t\geq 0,
\sigma_t=\inf\,\text{\rm support}(\calQ_t). \label{fme9}
\end{align}

 \end{definition}
Note that from \eqref{fme1}, \eqref{fme2} and \eqref{fme4}, it follows that
\begin{equation}\label{101}
X_t=X_0+E_t-D_t-R_t.
\end{equation}

The FME are based on analogy to the pre-limit dynamics.
Equation \eqref{fme5} is similar to \eqref{12}, where
the departure term is represented in terms of the hazard rate function.
The equations
\eqref{fme1}, \eqref{fme2}, \eqref{fme3} and \eqref{fme4} are analogous to \eqref{3}, \eqref{AB1}, \eqref{5}
and \eqref{7}. Also,
\eqref{fme7}, \eqref{fme8} and \eqref{fme9} are analogous to \eqref{10}, \eqref{6} and \eqref{11},
respectively. We see that \eqref{fme6}
relates the MVSP with data $(\al, K+R)$ and therefore analogy comes from \eqref{13}.

Define $\calE\in\C^{\uparrow, 0}_{\calM}$ by
$$\calE_t[0, x]=\int_0^t\ind_{\{x\geq s\}} \la_s \pi[0, x-s] ds, \quad \forall\; t,\; x\in [0, \iy).$$
Then $\al=\calQ_0+\calE$. Given $t_0\in [0, \iy)$, define for $s\in[0, \iy)$,
\begin{align}\label{18}
\calE^{[t_0]}_s[0, x]& =\int_0^s\la_{t_0+p}\ind_{\{x\geq p\}}\pi[0, x-p] dp,\quad \calQ^{[t_0]}_s[0, x] =\calQ_{t_0+s}[0, t_0+x],
\\
\nu^{[t_0]}_s &=\nu_{t_0+s}, \quad K^{[t_0]}_s=K_{t_0+s}-K_{t_0}, \quad R^{[t_0]}_s=R_{t_0+s}-R_{t_0}.\nonumber
\end{align}
The proof of the following time shift lemma is straightforward.
\begin{lemma}\label{time-shift}
Suppose that $(\calQ, \nu, K, R)$ is a solution to the FME
with data $(\calQ_0, \nu_0, \al, X_0)$. Then for any
$t_0\geq 0$, $(\calQ^{[t_0]}, \nu^{[t_0]}, K^{[t_0]}, R^{[t_0]})$
satisfies the FME
with data $(\calQ_{t_0}, \nu_{t_0},\calQ_{t_0}+\calE^{[t_0]} ,X_{t_0})$.
\end{lemma}

Let us also recall the following result from \cite[Theorem 4.1]{kaspi-ramanan}.
\begin{lemma}\label{lem3.2}
If $\nu$ satisfies \eqref{fme5} then for any $f\in \C_c(\R_+)$ we have
$$\int_{[0, H^s)} f(x) \nu_t(dx) = \int_{[0, H^s)} f (x+t)\frac{1-G(x+t)}{1-G(x)}\nu_0(dx) + \int_0^t f(t-s)(1-G(t-s))dK_s.$$
\end{lemma}

The first main result of this article states that solutions to the FME are unique.

\begin{theorem}\label{Thm1}
Let either Assumption~\ref{Assump3} or \ref{Assump4} hold.
Then given $(\calQ_0, \nu_0, \al, X_0)\in\calS_0$, there exists at most one
solution $(\calQ, \nu, K, R)$ to the FME.
\end{theorem}

Our second main result asserts that the FME characterize the limit.

\begin{theorem}\label{Thm3}
Suppose that Assumptions~\ref{Assump1}--\ref{Assump2} and either Assumption~\ref{Assump3}
or \ref{Assump4} hold. Then there exists
a unique solution $(\calQ, \nu, K, R)$ to the FME and the sequence $(\bar\calQ^N, \bar\nu^N, \bar{K}^N, \bar{R}^N)$ converges in distribution, as $N\to\iy$,
to $(\calQ, \nu, K, R)$.
\end{theorem}

We prove Theorem~\ref{Thm1} in Section~\ref{sec-fme}. The proof
is established by two propositions, Proposition~\ref{Thm4.4} and \ref{Thm4.5}.
Sections \ref{sec-tight} and \ref{sec-char} are devoted to prove Theorem~\ref{Thm3},
by showing tightness of the scaled processes in Section \ref{sec-tight},
and the characterization of limits in terms of the FME in Section~\ref{sec-char}.

In \cite[Remark 2.2]{atar-bis-kaspi}, the authors establish the fluid limit of $G/M/N+G$ queueing model governed by EDF
discipline where the patience time distribution $\pi$ is assumed to
have the property that $x\to\pi[0, x]$ is strictly increasing on the support of $\pi$. In view of Theorem~\ref{Thm1} and Theorem~\ref{Thm3} we have
the following corollary that generalizes the results in \cite{atar-bis-kaspi}.

\begin{corollary}\label{Cor1}
Let the service requirement distribution be given by an exponential distribution. Assume that Assumptions~\ref{Assump1} holds and $\nu_0$ does not have atoms.
Then $(\bar\calQ^N, \bar\nu^N, \bar{K}^N, \bar{R}^N)$ converges in distributions, as $N\to\iy$,
to $(\calQ, \nu, K, R)$ where $(\calQ, \nu, K, R)$ is the unique solution to the FME.
\end{corollary}

We note that when the patience time distribution is deterministic then G/G/N+D queueing models under EDF policy is same as the queueing model
under FCFS (or FIFO) scheduling policy. \cite{kang-ramanan} studies the fluid limit for G/G/N+G queueing systems working under FCFS discipline where
it assumes that the patience time distribution has density.
This is further generalized in \cite{zuniga} to a more general class of distributions
but with continuous patience time distribution.
Thus in particular, it does not cover the case of deterministic patience time. Since our
Assumption~\ref{Assump4} includes deterministic patience time we have,

\begin{corollary}\label{Cor2}
Let Assumption~\ref{Assump1}-\ref{Assump2} hold. Then for a queueing model G/G/N+D working under FCFS discipline we have LLN limit and the unique limit can be characterised by the FME \eqref{1000}-\eqref{fme9}.
\end{corollary}

\section{Uniqueness of solution to the FME}\label{sec-fme}
\beginsec

This section is devoted to the proof of uniqueness of solutions to the FME
\eqref{1000}--\eqref{fme9}, Theorem~\ref{Thm1}.
First we show the result under Assumption~\ref{Assump3} (Proposition~\ref{Thm4.4})
and then under Assumption~\ref{Assump4} (in Proposition~\ref{Thm4.5}).

\subsection{Uniqueness under Assumption~\ref{Assump3}}\label{uni-1}

Under Assumption~\ref{Assump3},
for any compact set $\bC\subset [0, H^s)\setminus\calZ$ we have
$\inf_{x\in\bC} h(x)>0$ where $h$ denotes the hazard rate function of $G$.
Now let us begin with the following lemma.

\begin{lemma}\label{lem1}
Let $(K, \nu)$ be continuous and satisfy \eqref{fme5}.
Furthermore, assume that the measure $\nu_0$ is atomless.
Then $\nu$ is atomless, and for any $\eps, T> 0$ there exists $\del>0$ such that
\begin{equation*}
\sup_{t\in[0, T]}\sup_{x\in\R_+}\nu_t[x-\del, x+\del]<\eps.
\end{equation*}
\end{lemma}

\begin{proof}
Fix $s\in [0, T]$ and consider $x\in\R_+$.
Since $\nu_0\in\calM^0$, the map $x\mapsto \nu_0((-\infty, x])$ is non-decreasing, bounded and continuous on $\R$,
and therefore uniformly continuous on $\R$. Hence
there exists $\del_1>0$ such that
\begin{equation}\label{01}
\sup_{x\in\R_+}\nu_0[x-\del_1, x+\del_1] <\frac{\eps}{4}.
\end{equation}
Let $\del<\del_1$ and
$f_\del\in \C_{b, +}(\R_+)$ be such that $f_\del=1$ on $[x-\frac{\del}{2}, x+\frac{\del}{2}]$,
$f_\del$ vanishes outside $[x-\del, x+\del]$, and $0\leq f_\del\leq 1$.
By Lemma~\ref{lem3.2},
$$
\lan f_\delta, \nu_s\ran \leq \sup_{x\in\R_+}\nu_0[x-\del_1, x+\del_1] + \text{osc}_{2\del} (K, T).
$$
Therefore using the fact that $K$ is continuous and \eqref{01} we have the result.\qed
\end{proof}

Now we state a uniqueness result under Assumption~\ref{Assump3}.
\begin{proposition}\label{Thm4.4}
Let Assumption~\ref{Assump3} hold.
Let $(\calQ^i, \nu^i, K^i, R^i), i=1,2,$ be two of solutions of the FME \eqref{1000}--\eqref{fme9} on $[0, T]$.
Then they are equal on $[0, T]$.
\end{proposition}

Let $\eps\in(0, 1/2)$. Then by tightness there exists a sequence of compact sets $\bC_n$ such that
$\max_{i=1,2}\sup_{s\in[0, T]}\nu_s^i(\bC^c_n)\to 0$
as $n\to\infty$.
Combining with Lemma~\ref{lem1} we find a compact set $\bC\subset [0, H^s)\setminus\calZ$ such that
\begin{align}\label{02}
\max_{i=1,2}\sup_{s\in[0, T]} \nu_s^i(\bC^c)<\eps.
\end{align}
Here we have used the fact that $\nu^i([0, H^s)^c)=0$ which is evident from Lemma~\ref{lem3.2}. Define $h_0=\inf_{x\in \bC}h(x)$. From
Assumption~\ref{Assump3}(1) we note that
$h_0$ is positive. Let $m=(1-\eps) h_0$ and $\del$
be such that $\sup_{s\in[0, T]} \la_s\pi[0, \del]< m$
and $G(\del)<\frac{1}{10}$.
We shall show that if the two solutions agree at $t<T$ then
they agree on the interval $[t, t+\del]$. This will prove the result.
A key observation is that when the solutions are translated using
Lemma \ref{time-shift}, the same compact set $\bC$ chosen above
works for a fixed $\eps$ and thus the choice of $\del$ remains the same on $[0, T]$.

\begin{lemma}\label{lem2}
$(\calQ^1, \nu^1, K^1, R^1)=(\calQ^2, \nu^2, K^2, R^2)$ on $[0, \del]$.
\end{lemma}

The proof of Lemma~\ref{lem2} is based on Lemmas \ref{lem3}--\ref{lem4} below.
Recall $D^i$ from \eqref{fme4}. Denote $\Del D= D^1-D^2$.

\begin{lemma}\label{lem3}
Define $\tilde{R}_s= R^2_s+D^2_s-D^1_s +\lan 1, \nu^2_s\ran -\lan 1, \nu^1_s\ran$.
Then the image $\tilde\calQ$ of the data $(\al, K^1+\tilde{R})$ under $\Th_1$
satisfies $\tilde\calQ=\calQ^2$.
\end{lemma}

\begin{proof}
By \eqref{fme4}, $K^1+\tilde{R}=K^2+R^2$. As a result,
$\tilde\calQ$ and $\calQ^2$ are images of the same data under the map $\Th_1$,
and so they are equal.
\qed
\end{proof}

\begin{lemma}\label{lem5}
If for some $t\in [0, \del)$ and some $z_1\in(0,\del-t]$ we have $\calQ^1_t[0, t+z_1]=0$
then $\calQ^1_s[0, t+z_1]=0$ for all $s\in [t, t+z_1)$.
\end{lemma}

\begin{proof}
Assume the contrary. Then there exist $t$ and $z_1$ as in the statement of the lemma,
and there exists $s\in [t, t+z_1)$ such that $\calQ^1_s[0, t+z_1]>0$.
Define $s_0=\sup\{s'\leq s: \calQ^1_{s'}[0, t+z_1]=0\}$.
From \eqref{fme6} and \cite[Proposition~2.8]{atar-bis-kaspi-ram} we have $\calQ\in\C_{\calM^0}$.
Since $\calQ$ is atomless
and continuous we have $s_0\in [t, s)$. Appealing to the 1-dimensional Skorohod map,
we have
$$
\calQ^1_{s'}[0, t+z_1]=\Gam_1(\calQ^1_{s_0}[0, t+z_1] +
\bar\calE[0, t+z_1]-\bar{K}^1-\bar{R}^1)(s'),
$$
where $\bar\calE_{s'}=\calE_{s'}-\calE_{s_0},\, \bar{K}^1_{s'}=K^1_{s'}-K^1_{s_0},\, \bar{R}^1_{s'}=R^1_{s'}-R^1_{s_0}$. Now by definition we have
$\calQ_{s'}>0$ on $(s_0, s]$. Therefore $\lan 1, \nu^1_{s'}\ran =1$ on $[s_0, s]$. Hence by \eqref{fme4} we have
\begin{equation}\label{07}
\bar{K}^1_{s'}=\int_{s_0}^{s'}\lan h, \nu^1_p\ran dp.
\end{equation}
We note that when $\lan 1, \nu^1_{p}\ran =1$,
we have $\lan h, \nu^1_p\ran \geq h_0 \nu^1_p(\bC)>h_0(1-\eps)=m$.
Again
$$
\bar\calE_{s'}[0, t+z_1]=\int_{s_0}^{s'}\ind_{\{t+z_1\geq p\}} \la_p \pi[0, t+z_1-p]dp\leq \int_{s_0}^{s'} \sup_{[0, T]}\la_s \pi[0, \del]dp.
$$
Therefore $\bar\calE[0,t+z_1]-\bar{K}^1$ is non-increasing on $[s_0, s]$. Since $\calQ^1_{s_0}[0, t+z_1]=0$ we have $\calQ^1_{s}[0, t+z_1]=0$ which
is a contradiction. Hence follows the lemma.
\qed
\end{proof}

\begin{lemma}\label{lem4}
For all $s\in[0, \del]$ we have $|R^1_s-R^2_s|\leq \|D^1-D^2\|_s$.
\end{lemma}

\begin{proof}
Define $\tau =\inf\{t :\ R^1_t> R^2_t +\|D^1-D^2\|_t\}$.
If $\tau\geq \del$ then there is nothing to prove. Arguing
by contradiction, assume that $\tau<\del$. We claim that there exists $t_1<\del$ such that
\begin{align*}
R^1_{t_1}  > R^2_{t_1} + \|D^1-D^2\|_{t_1},\quad
\text{and, for any neighbourhood $\boldsymbol{O}$ of $t_1$ we have} \int_{\boldsymbol{O}}dR^1>0.
\end{align*}
To prove this claim we chose $t_2\in (\tau, \del)$ such that
$R^1_{t_2}  > R^2_{t_2} + \|D^1-D^2\|_{t_2}$. Define $t_1=\inf\{ t\leq t_2, R^1_t=R^1_{t_2}\}$.
By \eqref{1000}, $E_0=0$, and thus by
\eqref{fme1}, $R^1_0=0$.
Hence we get $t_1>0$ as $R^1_{t_1}=R^1_{t_2}>0$. Again
\begin{equation}\label{100}
R^1_{t_1}= R^1_{t_2} > R^2_{t_2} +\|D^1-D^2\|_{t_2}\geq
R^2_{t_1} +\|D^1-D^2\|_{t_1}.
\end{equation}
Also by definition of $t_1$, we have for any $t<t_1$ that $R^1_{t_1}-R^1_t>0$. This establishes the claim we made above. Therefore by \eqref{fme8} we have a sequence
$\{s_n\}, s_n\to t_1$, such that $X^1_{s_n}\geq 1$. Therefore $\lan 1, \nu^1_{s_n}\ran =1$ for all $n$, implying by continuity $\lan 1, \nu^1_{t_1}\ran =1$. Define
$$
\Del_0=R^1_{t_1}-R^2_{t_1}-\|D^1-D^2\|_{t_1}-\lan 1, \nu^2_{t_1}\ran + \lan 1, \nu^1_{t_1}\ran.
$$
Then by \eqref{100}, $\Del_0>-\lan 1, \nu^2_{t_1}\ran + \lan 1, \nu^1_{t_1}\ran$.
Since $\lan 1, \nu^2_{t_1}\ran \leq 1$ we have $\Del_0 >0$. Since the measures $\calQ_0, \calE$ are atomless we can find $y=t_1-\eps_1, \, z=t_1+\eps_1$
such that $t_1+\eps_1<\del$ and
$$ \calQ_0(y, z] +\calE_{t_1}(y, z] <\frac{\Del_0}{2}.$$
Define
\begin{align*}
\psi^1_s(z) &= \calQ_0[0, z] + \calE_s[0, z] - K^1_s- R^1_s,
\\
\tilde{\psi}_s(y) & = \calQ_0[0, y] + \calE_s[0, y] - K^1_s-\tilde{R}_s,
\end{align*}
where $\tilde R$ is  same as in Lemma \ref{lem3}.
From \eqref{fme6} we have $\calQ^1_{t}[0, z]=\Gam_1[\psi^1(z)](t)$ and by \eqref{fme7} and Lemma~\ref{lem3}  we have $0=\calQ^2_{t_1}[0, y]=\Gam_1[\tilde\psi(y)](t_1)$.
Now
\begin{align*}
\psi^1_{t_1}(z)&=\tilde\psi_{t_1}(y) + \calQ_0(y, z] + \calE_{t_1}(y, z] -(R^1_{t_1}-\tilde{R}_{t_1})\\
&\le\tilde\psi_{t_1}(y)+\frac{\Del_0}{2}-\Big(R^1_{t_1}-R^2_{t_1}-(D^2_{t_1}-D^1_{t_1})
-\lan 1,\nu^2_{t_1}\ran+\lan1,\nu^1_{t_1}\ran\Big).
\end{align*}
Noting that $D^2_{t_1}-D^1_{t_1}\le\|D^2-D^1\|_{t_1}$, it follows from the definition
of $\Del_0$ that
\[
\psi^1_{t_1}(z)\leq \tilde\psi_{t_1}(y) + \frac{\Del_0}{2}-\Del_0
= \tilde\psi_{t_1}(y) - \frac{\Del_0}{2}.
\]
Since $\Gam_1[\tilde\psi(y)](t_1)=0$ we have $\tilde\psi_{t_1}(y)\leq 0$.
Let $t_0=\inf\{ t\geq 0 : \psi^1_t(z)\leq 0\}$.
By the above observation, using the continuity of $t\mapsto\psi^1_t(z)$,
we have $t_0<t_1$.
Thus by the definition of $t_0$,
$\psi^1_{t_0}(z)=\inf_{[0, t_0]}(\psi^1_s(z)\wedge 0)$ implying
$\calQ^1_{t_0}[0, z]=\Gam_1(\psi^1(z))(t_0)=0$ and $t_0<t_1<z$.
Now by Lemma~\ref{lem5} we have $\{\sigma^1_t > t\}$ on $[t_0, t_1]$.
But $dR^1$ charges the interval $[t_0, t_1]$.
This is a contradiction to \eqref{fme9} and thus $\tau\geq \del$.
\qed
\end{proof}

\vspace{.2in}

\noi{\bf Proof of Lemma~\ref{lem2}:} Denote $\Del K = K^1-K^2$.
Analogously, define other quantities as $\Del R, \Del Q, \Del X$. Then from
\eqref{fme1} and \eqref{fme3} we have on $[0, \del]$,
\begin{align*}
|\Del K|\leq |\Del Q| + |\Del R| &\leq |\Del X| + |\Del R|
\\
&\leq 2 |\Del R| + |\Del D|
\\
&\leq 3 \|D^1-D^2\|_\del,
\end{align*}
where in the third inequality we used \eqref{101},
and in the last one we used Lemma~\ref{lem4}. Thus
\begin{equation}\label{08}
\|\Del K\|_\del\leq 3 \|D^1-D^2\|_\del.
\end{equation}
By \cite[Corollary 4.4]{kaspi-ramanan}, we have
$$ D^i_t = \int_{[0, H^s)} \frac{G(x+t)-G(x)}{1-G(x)}\nu_0(dx) + \int_0^t g(t-s) K^i_s ds , \quad i=1,2.$$
Therefore, for $t\le\del$,
$$\Del D_t = \int_0^t g(t-s) \Del K_s ds,$$
and
\begin{align*}
|\Del D_t| &\le  \int_0^t g(t-s) |\Del K_s| ds
\\
&\le \|\Del K\|_tG(t)\\
&\le \|\Del K\|_\del G(\del).
\end{align*}
Hence $\|\Del D\|_\del\le \|\Del K\|_\del G(\del)\le\frac{1}{10}\|\Del K\|_\del$.
Therefore by \eqref{08}, $\|\Del K\|_\del\le \frac{3}{10}\|\Del K\|_\del$.
This shows $\Del K=0$ on $[0,\del]$. Hence $\Del D=0$ on $[0,\del]$.
Using Lemma \ref{lem4}, also $\Del R=0$ on $[0,\del]$.
The equality of $\nu^1=\nu^2$ on $[0, \del]$ now follows from Lemma~\ref{lem3.2}.
Finally, $\calQ^1=\calQ^2$ follows from the uniqueness of solutions to the MVSP.
\qed

\skp

\noindent{\bf Proof of Proposition~\ref{Thm4.4}.}
To prove uniqueness on $[0, T]$ we argue that if the solutions agree at time $t<T$
then they also agree
on $[t, (t+\del)\wedge T]$ where $\del$ is as in Lemma~\ref{lem2}.
Recall $(\calQ^{[t]}, \nu^{[t]}, K^{[t]}, R^{[t]})$ from \eqref{18}.
By Lemma~\ref{time-shift}
$(\calQ^{[t]}, \nu^{[t]}, K^{[t]}, R^{[t]})$ satisfies the FME with given data
$(\calQ_t, \nu_t, \calE^{[t]}, X_t)$ on $[0, T-t]$.
Hence uniqueness on $[t, (t+\del)\wedge T]$ reduces to
uniqueness on $[0, (T-t)\wedge\del]$ with the given data.
The latter follows from Lemma~\ref{lem2}. This completes the proof.
\qed

\subsection{Uniqueness under Assumption~\ref{Assump4}}\label{uni-2}

Recall from Assumption~\ref{Assump4} that $\pi[0, \kappa_1]=0$.
Without loss of generality we assume $\kappa_1=2$. We will show that any two
solutions agree on $[0, 1]$ provided they agree at $t=0$.
Then the uniqueness on any interval follows by applying Lemma~\ref{time-shift}. Let
\begin{equation}\label{20}
 H(x, t) = \calQ_0(x, \iy) + \int_0^t\la_s \ind_{\{x\geq s\}} \pi(x-s, \iy)ds.
\end{equation}
We invoke the Skorohod problem with time-varying boundary. We set the boundary function to be
$b_t=H(t, t)$ for $t\geq 0$.
\begin{definition}
Given $\psi\in\C_{\R_+}$, a pair $(\phi, \eta)\in (\C_{\R_+})^2$, is said to solve the Skorohod problem on the time varying domain $(-\iy, b]$, if
\begin{enumerate}
\item[1.] $\phi_t=\psi_t-\eta_t, \quad \text{for all}\;\; t\geq 0$,

\item[2.] $\phi_t\leq b_t$ for all $t\geq 0$,
\item[3.] $\eta$ is non-negative, non-decreasing and $\int_0^\cdot\ind_{\{\phi_s<b_s\}}d\eta_s=0$.
\end{enumerate}
\end{definition}
It is known that there is a unique solution $(\phi, \eta)$ given data $\psi$
\cite{burdzy-kang-ramanan}.
It has also been observed in \cite{atar-bis-kaspi} that this type
of Skorohod problem shows up in the fluid limits of G/G/1+G queueing models with EDF scheduling.
We denote the above Skorohod map by $\Gam^b$, so that $\Gam^b_1[\psi]=\phi$ and $\Gam^b_2[\psi]=\eta$.
Let $(\calQ, \nu, K, R)$ be a solution to the FME \eqref{1000}--\eqref{fme9}.
Define
$$
\varrho=\inf\{t\geq 0\; : Q_t\leq H(3/2, t)\}\wedge 1.
$$

\begin{lemma}\label{lem4.9}
If $(\calQ, \nu, K, R)$ is a solution to the FME \eqref{1000}--\eqref{fme9} then
\begin{align}
R_s & = \Gam^b_2[Q_0+E-K](s)\quad \text{for}\; \; s\in[0, \varrho),\label{21}
\\
R_1-R_{\varrho} & =0.\label{22}
\end{align}
\end{lemma}

\begin{proof}
First we note that for $t\in [0, 1]$,
$$ b_t=H(t,t)= \calQ_0(t, \iy) + \int_0^t\la_s\ind_{\{t\geq s\}}\pi(t-s, \iy)=\calQ_0(t, \iy)+E_t,$$
where we used the fact that $\pi[\kappa_1, \infty)=\pi[2, \iy)=1$. Since $\calQ_t$ does not have atoms
and $\calQ_t[0, t)=0$ by \eqref{fme7} we have for $t\in[0, 1]$,
\begin{align}\label{23}
Q_t=\calQ_t(t, \iy) &=Q_t-\calQ_t[0, t]\nonumber
\\
&= Q_0 + E_t-K_t-R_t-\Gam_1[\al[0, t]-K-R](t)\nonumber
\\
&= Q_0 + E_t-K_t-R_t-(\al[0, t]-K_t-R_t)-\Gam_2[\al[0, t]-K-R](t)\nonumber
\\
&\leq Q_0 + E_t- \calQ_0[0, t]-\calE_t[0, t]=b_t,
\end{align}
where we have used the fact that $\Gam_2$ is non-negative valued and $\calE_t[0, t]=0$.
To show \eqref{21} we may assume $\varrho>0$, as otherwise there is nothing to prove.
Thus on $[0, \varrho)$, $Q_s> H(3/2, s)$ and
$$
H(3/2, s)= \calQ_0(3/2, \iy) +E_s.
$$
For $s\in [0, \varrho)$, we get from \eqref{fme1} that
\begin{align}\label{AB2}
0>H(3/2, s)-Q_s = \calQ_0(3/2, \iy) + E_s -\Big(Q_0+E_s-K_s-R_s\Big) =K_s+R_s-\calQ_0[0, 3/2].
\end{align}
From \eqref{AB2} we also note that if for some $t$ we have
$K_t+R_t-\calQ_0[0,3/2]<0$ then $K_s+R_s-\calQ_0[0,3/2]<0$ for all $s\leq t$ since
$K+R$ is non-decreasing.
Therefore we have $\varrho=\sup\{s\in[0, 1]\; :\; K_s+R_s< \calQ_0[0, 3/2]\}$.
Since $\varrho>0$ we have $\calQ_0[0, 3/2]>0$. Define
$$\tilde\sigma_t=\sup\{x\; :\; K_t+R_t\geq \calQ_0[0, x]\}.$$
We claim that for $t\in[0,\varrho), \tilde{\sigma}_t$ is equal to the infimum of the support of $\calQ_t$, namely,  $\sigma_t$ .  Since $\calQ_0$ is atomless,
\begin{equation}\label{24}
K_{t}+ R_{t}= \calQ_0[0, \tilde\sigma_t], \quad \text{for}\;\; t\in [0, \varrho).
\end{equation}
Fix $t\in[0, \varrho)$. It is enough to show that for any $x<\tilde\sigma_t$,
we have $\calQ_t[0, x]=0$ and for $x>\tilde\sigma_t$, we have
$\calQ_t[0, x]>0$. Now $t<\varrho$ implies that $Q_t> H(3/2, t)= \calQ_0(3/2, \infty) + E_t$, and
therefore by \eqref{fme1} we have $\calQ_0[0, 3/2]> K_t + R_t$. This implies
$\tilde\sigma_t<3/2$.
Hence it suffices to pick $x$ from $[0, 3/2]$. Take
$x\in [0, 3/2]$ and use \eqref{fme6} to obtain
\begin{equation}\label{25}
\calQ_t[0, x]=\calQ_0[0, x]  -K_t-R_t +\sup_{s\leq t}\Big(K_s+R_s-\calQ_0[0, x]\Big)^+.
\end{equation}
Since $K+R$ is non-decreasing, we see from \eqref{25} that for any $x<\tilde\sigma_t$, $\calQ_t[0, x]=0$. Again for $x> \tilde\sigma_t$ we have
$\sup_{s\leq t}\Big(K_s+R_s-\calQ_0[0, x]\Big)^+=0$ and therefore $\calQ_t[0, x]>0$. Thus the claim follows. Now for $t\in[0, \varrho)$,
\begin{align*}
b_t-Q_t= H(t, t)-Q_0-E_t+K_t+R_t &=\calQ_0(t, \iy)+E_t-Q_0-E_t +\calQ_0[0, \sigma_t]
\\
&= \calQ_0(t, \iy)-\calQ_0(\sigma_t, \iy),
\end{align*}
where we use \eqref{24}. Thus $b_t-Q_t>0$ implies $\sigma_t> t$ for $t\in [0, \varrho)$. Therefore using \eqref{fme9} we have
for $t\in [0, \varrho)$,
\begin{equation}\label{26}
\int_0^t\ind_{\{Q_s<b_s\}}\, dR_s\leq \int_0^t\ind_{\{\sigma_s>s\}}\, dR_s =0.
\end{equation}
Therefore \eqref{21} follows from \eqref{fme1}, \eqref{23} and \eqref{26}. Now we prove \eqref{22}. Without loss of generality, we assume $\varrho<1$,
otherwise there is nothing to prove. Since $\calQ_\varrho$ is in $\calM^0$ by \cite[Proposition~2.8]{atar-bis-kaspi-ram},
 \eqref{25} and \eqref{AB2} gives us
$\calQ_\varrho[0, 3/2]=\lim_{t\uparrow \rho} \calQ_t[0, 3/2]=0$. Thus by definition
$\sigma_\varrho=\inf\,\text{support}(\calQ_\varrho)\geq 3/2$. Therefore from
\eqref{theta-gamma} and \eqref{fme6} we have
$K_\varrho+ R_\varrho\geq \calQ_0[0, 3/2]$.
Now for any $t\in[\varrho, 1]$ and using \eqref{25} we get
\begin{align*}
\calQ_t[0, 3/2] &=\calQ_0[0, 3/2] -K_t-R_t+\sup_{s\leq t}\big(K_s+R_s-\calQ_0[0, 3/2]\big)^+
\\
&= \calQ_0[0, 3/2] -K_t-R_t + (K_t+R_t-\calQ_0[0, 3/2]\big) =0.
\end{align*}
Thus for $t\in[\varrho, 1]$ we have $\sigma_t\geq 3/2>t$. Hence from \eqref{fme9} we obtain $R_1-R_\varrho=0$.\qed
\end{proof}

\begin{proposition}\label{Thm4.5}
Let Assumption~\ref{Assump4} hold. If $(\calQ^i, \nu^i, K^i, R^i), i=1,2,$
solve the FME \eqref{1000}--\eqref{fme9} on $[0, T]$
with data $(\calQ_0, \nu_0, \al, X_0)$
then $(\calQ^1, \nu^1, K^1, R^1)=(\calQ^2, \nu^2, K^2, R^2)$ on $[0, T]$.
\end{proposition}

\begin{proof}
As we commented earlier, we show the uniqueness
on $[0, 1]$ (keep in mind that $\kappa_1=2$) and then one can apply Lemma~\ref{time-shift}
to extend the time interval. We distinguish between two cases.

\noindent{\bf Case 1.} $\calQ_0[0, 3/2]>0$. Then $\varrho>0$ where $\varrho$ is
as in Lemma~\ref{lem4.9}.
Let $\varrho^i$ be the time corresponding to the $i$-th system.
As shown in Lemma~\ref{lem4.9}, we have that
for $s\in[0, \varrho^i)$, $K^i_s+R^i_s<\calQ_0[0, 3/2]$.
Thus for $i=1,2$, on $[0, \varrho^i)$, we have
\begin{equation}\label{27}
Q^i_s= Q_0 +E_s-K^i_s-R^i_s>\calQ_0(3/2, \iy) + E_s\geq 0.
\end{equation}
Thus the system has a busy period in $[0, \varrho^i]$.
We apply \cite[Corollay 4.4]{kaspi-ramanan} by which
\begin{equation}\label{28}
K^i_s=\mean{1, \nu^i_s}-\mean{1, \nu_0} + \int_{[0, H^s)}\frac{G(x+s)-G(x)}{1-G(x)}\nu_0(dx) + \int_0^t g(t-s)K^i_s ds.
\end{equation}
On $[0, \varrho^i]$ we have $\mean{1, \nu^i_s}=1$ by \eqref{27} and therefore from \eqref{28}
we obtain
\begin{align*}
K^i_s= \int_{[0, H^s)}\frac{G(x+s)-G(x)}{1-G(x)}\nu_0(dx) + \int_0^t g(t-s)K^i_s ds, \quad \text{for}\; \; s\in [0, \varrho^i].
\end{align*}
However, the above is a renewal equation and therefore has a unique solution
\cite[Th 5.2.4]{asmussen}
%Asmussen Th 2.4, chapter 5.(page 146).
%\cite[Corollary~4.4]{kaspi-ramanan}
Therefore $K^1=K^2$ on $[0, \varrho^1\wedge\varrho^2]$. Thus by Lemma~\ref{lem3.2} we have $\nu^1=\nu^2$ on $[0, \varrho^1\wedge\varrho^2]$.
Applying Lemma~\ref{lem4.9} we obtain
$(\calQ^1, \nu^1, K^1, R^1)=(\calQ^2, \nu^2, K^2, R^2)$ on $[0, \varrho^1\wedge\varrho^2]$.
Again, this would contradict the definition of $\varrho$ unless
we have $\varrho^1=\varrho^2$.
It remains to prove that the equality holds on $[\varrho, 1]$.
We already know from Lemma~\ref{lem4.9} that
$R^i_1-R^i_{\varrho}=0$ for $i=1,2$.
Thus to show equality we only need to show that
$(\nu^1, K^1)=(\nu^2, K^2)$ on $[\varrho, 1]$.
To this end, we consider
the shifted solutions $(\nu^{i,[\varrho]}, K^{i,[\varrho]})$
for $i=1,2$, defined as in \eqref{18}.
Since on $[\varrho, 1]$ the system behaves
like a system without reneging the equality follows from \cite[Theorem 4.6]{kaspi-ramanan}.

\noindent{\bf Case 2.} Let $\calQ_0[0, 3/2]=0$. In this case $\varrho^i=0$,
and by Lemma~\ref{lem4.9}, $R^i_1=0$ for $i=1,2$. The equality now follows
by similar arguments to those of case 1.
\qed
\end{proof}

\section{Tightness of the scaled processes}\label{sec-tight}
\beginsec

In this section we establish tightness of the scaled processes defined in Section~\ref{main-result}. First we introduce a family of martingales that plays
a crucial role in this proof (following ideas of \cite{kaspi-ramanan, kang-ramanan}). For any bounded measurable function $\varphi$ defined on $[0, H^s)\times\R_+$, let
\begin{equation}\label{30}
A^N_t(\varphi) = \int_0^t\int_{[0, H^s)}\varphi(x, s)h(x)\nu^N_s(dx)\, ds.
\end{equation}

\begin{lemma}\label{lem5.1}
For every bounded measurable function $\varphi$ defined on $[0, H^s)\times\R_+$,
the process $\calM^N(\varphi)$ defined by
\begin{equation}\label{31}
 \calM^N(\varphi)= D^N(\varphi)-A^N(\varphi),
\end{equation}
where $D^N(\varphi)$ is given by \eqref{9},
is a local $\calF^N_t$-martingale. Moreover, for every $N\in\N,\, t\in[0, \iy)$ and $c\in [0, H^s)$,
\begin{equation}\label{32}
\big| A^N_t(\varphi)\big| \leq \|\varphi\|_{\iy}(X^N_0 + E^N_t)\int_0^ch(x)dx < \iy,
\end{equation}
for any $\varphi\in\C_c([0, H^s)\times\R_+)$ with support$(\varphi)\subset[0, c]\times\R_+$. In addition, the
predictable quadratic variation process $\mean{\bar\calM^N(\varphi)}$
of $\bar\calM^N(\varphi)=\frac{1}{N}\calM^N(\varphi)$ satisfies
\begin{equation}\label{33}
\lim_{N\to\iy}\sup_{s\in[0, t]}\E[\mean{\bar\calM^N_s(\varphi)}]=0, \quad \bar\calM^N(\varphi)\Ra 0,
\end{equation}
for every $\varphi\in\C_b([0, H^s)\times\R_+)$.
\end{lemma}

\begin{proof}
The proof of the local martingale property of $\calM^N(\varphi)$ follows using the argument
in \cite[Lemma 5.4 and Corollary 5.5]{kaspi-ramanan}. The filtration used in \cite{kaspi-ramanan} is
smaller than the one considered here.
We have one added element in our filtration $\{\calF^N_t\}$, corresponding to the deadline information. But the deadlines $u^N$ are independent
of the service requirement process. Therefore one can apply conditional independence
in a straightforward fashion to obtain the local martingale property of $\calM^N(\varphi)$. A similar argument by independence
is also used in \cite[Proposition 5.1]{kang-ramanan}, where one can check
for the calculations.
The proof of \eqref{32} can be established following
\cite[Proposition 5.7]{kaspi-ramanan}. For the proof of \eqref{33} we refer to \cite[Lemma~5.9]{kaspi-ramanan}.
\qed
\end{proof}

\skp

A sequence of processes with sample paths in $\D_\calS$, $\calS$ being a Polish
space, is said to be {\it $\C$-tight}
if it is tight in $\D_\calS$ (in the $J_1$ topology) and, in addition,
any subsequential limit has, a.s., paths in $\C_\calS$. The following
characterization will be useful.

\skp

\noi
{\bf Characterization of $\C$-tightness for processes with sample paths in
$\D_{\R}$} (see
Proposition VI.3.26 of \cite{jacod-shiryaev}):
{\it
$\C$-tightness of a sequence of processes $Z^N$ is equivalent to
\begin{itemize}
\item[C1.] The sequence of random variables
$\|Z^N\|_T$ is tight for every fixed $T<\iy$, and
\item[C2.] For every $T<\iy$,
$\eps>0$ and $\eta>0$ there exist $N_0$ and $\theta>0$ such that
\[
N\ge N_0 \text{ implies } \p(\osc_\theta(Z^N,T)>\eta)<\eps.
\]
\end{itemize}
}

\begin{lemma}\label{lem5.2}
Let Assumption~\ref{Assump1} hold. Then,
for $\bar{Z}^N=\bar{K}^N, \bar{X}^N, \bar{R}^N, \mean{1, \bar{\nu}^N}$, the sequence $\{\bar{Z}^N\}$
and the sequences $\{\bar{D}^N(\varphi)\}, \{\bar{A}^N(\varphi)\}$, for every $\varphi\in\C_{b, +}([0, H^s)\times\R_+)$, are $\C$-tight.
\end{lemma}

\begin{proof}
Fix $T\in (0, \iy)$. By Lemma~\ref{lem5.1} we have
for some increasing
sequence of stopping times $\hat\tau_n$, $\hat\tau_n\to\infty,$ that
$\E[\bar{A}^N_{T\wedge\hat\tau_n}(\varphi)]=\E[\bar{D}^N_{T\wedge\hat\tau_n}(\varphi)]$. Therefore
by monotone convergence, $\E[\bar{A}^N_{T}(\varphi)]=\E[\bar{D}^N_{T}(\varphi)]$.
Using the final assertion of Assumption \ref{Assump1},
\begin{equation}\label{34}
\sup_N\E[\bar{A}^N_T(\varphi)]=\sup_N\E[\bar{D}^N_T(\varphi)]\leq \|\varphi\|_\iy
\sup_N\E[\bar{X}^N_0 + \bar{E}^N_T]<\iy.
\end{equation}
Similarly, using \eqref{3}--\eqref{5} we obtain
\begin{align}\label{35}
\sup_{N}\E[\bar{U}^N_T]
& \leq \sup_N\E[\bar{X}^N_0 + \bar{E}^N_T]<\iy,
\\
&\quad\,\quad \text{for}\;
\bar{U}^N_T=\bar{K}^N_T, \bar{R}^N_T, \sup_{[0, T]}\mean{1, \bar{\nu}^N},
\sup_{[0, T]}\bar{X}^N.\nonumber
\end{align}
Since $\bar{A}^N_T(\varphi), \bar{D}^N_T(\varphi), \bar{K}^N, \bar{R}^N$ are non-decreasing for $\varphi\in\C_{b, +}([0, H^s)\times\R_+)$, criterion
$C1$ follows from \eqref{34} and \eqref{35}. Now we show that criterion C2 also holds. We start with $\bar{R}^N$.
Fix $\eps, \eta\in (0, \iy)$. Let $\theta>0$ and $s, t\in[0, T],s<t$
be such that $t\leq s+\theta$. We estimate $R^N_t-R^N_s$. Recall $\calE^N$ from \eqref{00}. Then
\begin{align*}
R^N_t-R^N_s &\leq \calQ^N_s[s, t] +\calE^N_t[s, t]-\calE_s^N[s, t]
\\
& \leq \al^N_s[s, t] + \calE^N_t[t-\theta, t].
\end{align*}
By Assumption~\ref{Assump1} we have $\sup_{t\in[0, T]}d_\calM(\bar\al^N_t, \al_t)\Ra 0$
and $\sup_{t\in[0, T]} d_\calM(\bar\calE^N_t, \calE_t)\Ra 0$ as $N\to\iy$.
Since $\al\in\C^{\uparrow, 0}_{\calM}$, we see that $x\mapsto \al_t[0, x]$ is uniformly continuous on $\R_+$, uniformly with respect to $t\leq T$.
On the other hand we also have
$$d_\calM(\bar\al^N_t, \al_t)\leq \sup_{x\in[0, \infty)}|\bar\al^N_t[0, x]-\al_t[0, x]| \leq d_\calM(\bar\al^N_t, \al_t) +
\osc_{2 d_\calM(\bar\al^N_t, \al_t)}(\al_t[0, \cdot], T).$$
Similar fact also holds true for $\bar\calE^N, \calE$. Therefore we can find $\theta>0$ so that
$$\lim_{N\to\iy}\p(\sup_{s\in[0, T]}\bar\al^N_s[s, s+\theta] + \sup_{t\in[0, T]}\bar\calE^N_t[t-\theta, t]>\eta)=0.$$
Combining the above two displays we see that $\bar{R}^N$ satisfies criterion C2.
Now, from \cite[Lemma~5.8(2), Lemma~5.12]{kaspi-ramanan} we obtain that,
for any $\varphi\in\C_{b, +}([0, H^s)\times\R_+)$,
\begin{equation}\label{36}
\lim_{\del\to 0}\limsup_{N\to\iy}\E\Big[\sup_{t\in[0, T]}(\bar{A}^N_{t+\del}(\varphi)-\bar{A}^N_{t}(\varphi))\Big]=0.
\end{equation}
Combining \eqref{36} with \eqref{33} we see that both the sequences $\{\bar{D}^N(\varphi)\}, \{\bar{A}^N(\varphi)\}$ satisfy C2 for
every $\varphi\in\C_{b, +}(\R_+)$. In particular, $\{\bar{D}^N(1)\}=\{\bar{D}^N\}$
also satisfies C2. Since
\begin{align*}
|\bar{X}^N_t-\bar{X}^N_s|& \leq |\bar{E}^N_t -\bar{E}^N_s|+|\bar{D}^N_t-\bar{D}^N_s|+|\bar{R}^N_t-\bar{R}^N_s|,
\\
|\mean{1, \bar{\nu}^N_t}-\mean{1, \bar{\nu}^N_s}| &\leq |\bar{X}^N_t-\bar{X}^N_s|,
\end{align*}
using \eqref{4} and \eqref{5}, we see that $\{\bar{X}^N\}$ and $\{\mean{1, \bar\nu^N}\}$
also satisfy C2. finally, the sequence $\bar{K}^N$ satisfies C2 by \eqref{7}.
\qed

\end{proof}

Recall the metric $d_\calM$ on $\calM$ from our notation.
\begin{lemma}\label{lem5.3}
Suppose that Assumption~\ref{Assump1} holds. For every $\eps, \eta$ and $T\in (0, \iy)$
there exist $\del\in(0, \iy)$ and $N_0$ such that
$$\text{for all}\;\; N\geq N_0, \quad \p\Big(\sup_{0\leq s\leq t\leq s+\del\leq T } d_\calM(\bar\nu^N_s, \bar\nu^N_t)>\eta\Big)<\eps. $$
\end{lemma}

\begin{proof}
Recall the definition \eqref{1} of $\nu^N_t$.
Let $F\subset\R_+$ be any closed set and denote by $F^{\eps_1}$ its $\eps_1$-enlargement
in $\R_+$, for a given $\eps_1\in(0, \eta)$. Then
for $s\leq t\leq s+\eps_1/2$ we have from \eqref{1}
\begin{align}\label{38}
&\bar\nu^N_t(F)-\bar\nu^N_s(F^{\eps_1})\\
&\quad= \frac{1}{N}\sum_{j=-X^N_0+1}^{E^N_t}\ind_F(\om^N_j(t))\ind_{\{\om^N_j(t)<v_j,\ t\ge\gamma^N_i\}}
-\frac{1}{N}\sum_{j=-X^N_0+1}^{E^N_s}\ind_{F^{\eps_1}}(\om^N_j(s))\ind_{\{\om^N_j(s)<v_j, \ s\ge\gamma^N_i\}}\nonumber
\\
&\quad\leq \frac{1}{N}\sum_{j=-X^N_0+1}^{E^N_s}\Big(\ind_F(\om^N_j(t))\ind_{\{\om^N_j(t)<v_j, \ s\ge\gamma^N_i\}}-\ind_{F^{\eps_1}}(\om^N_j(s))
\ind_{\{\om^N_j(s)<v_j, \ s\ge\gamma^N_i\}}\Big)
+ \bar{K}^N_t-\bar{K}^N_s\nonumber
\\ \notag
&\quad\leq  \bar{K}^N_t-\bar{K}^N_s,
\end{align}
where the last inequality follows from the following fact: if a customer $j$ has entered service by time $s$ and receives service at time $t$ with $\om^N_j(t)\in F$,
then by definition $\om^N_j(s)\in F^{\eps_1}$ as $t-s\leq \eps_1/2$ and $\om^N_j$ grows linearly. Also
\begin{align}\label{39}
&\bar{\nu}_s(F)-\bar{\nu}_t(F^{\eps_1})\\
&\quad\leq \frac{1}{N}\sum_{j=-X^N_0+1}^{E^N_s}\ind_{F}(\om^N_j(s))\ind_{\{\om^N_j(s)<v_j, \ s\ge\gamma^N_i\}}
-\frac{1}{N}\sum_{j=-X^N_0+1}^{E^N_s}\ind_{F^{\eps_1}}(\om^N_j(t))\ind_{\{\om^N_j(t)<v_j, \ s\ge\gamma^N_i\}}\nonumber
\\
&\quad= \frac{1}{N}\sum_{j=-X^N_0+1}^{E^N_s}\Big(\ind_{F}(\om^N_j(s))\ind_{\{\om^N_j(s)<v_j, \ s\ge\gamma^N_i\}}-
\ind_{F^{\eps_1}}(\om^N_j(t))\ind_{\{\om^N_j(t)<v_j, \ s\ge\gamma^N_i\}}\Big)\nonumber
\\ \notag
&\quad\leq \bar{D}^N_t - \bar{D}^N_s,
\end{align}
where the last inequality follows from the fact that if a customer is in service at time $s$ but not at time $t$ then it must have completed its service in
the time interval $(s, t]$. Now using Lemma~\ref{lem5.2} we have $\theta\in(0, \iy), N_0\in\N$ such that
\begin{equation}\label{40}
\text{for all}\;\; N\geq N_0, \quad \p(\osc_\theta(\bar{Z}, T)>\eps_1)<\eps/2, \quad \text{for}\;\; \bar{Z}=\bar{K}^N, \bar{D}^N.
\end{equation}
Since $\osc_\theta$ is
increasing in $\theta$ we can chose $\theta\in(0, \eps_1/2)$.
Thus combining \eqref{38}, \eqref{39} and \eqref{40} we obtain
$$\text{for all}\;\; N\geq N_0, \quad \p\Big(\sup_{0\leq s\leq t\leq t+\del\leq T } d_\calM(\bar\nu^N_s, \bar\nu^N_t)>\eta\Big)<\eps, $$
for $\del=\theta$. This completes the proof.
\qed
\end{proof}

\begin{lemma}\label{lem5.4}
Suppose that Assumption~\ref{Assump1} holds. Then the sequence $\{\bar\nu^N\}$ is $C$-tight.
\end{lemma}

\begin{proof}
First we argue that tightness holds for $\{\bar{\nu}^N\}$, then $\C$-tightness.
As far as tightness is concerned,
since we have already established the oscillation bound stated in Lemma~\ref{lem5.3},
to show tightness we only need to show
the compact containment property (see \cite[Corollary 3.7.4]{ethier-kurtz}),
i.e., that for each $T, \eta>0$, there
exists a compact set $\mathbf{K}_{T,\eta}\subset\calM$ such that
\[
\liminf_{N\ra\iy}\p\left(\bar\nu_t^N\in\mathbf{K}_{T,\eta} \mbox{ for all }
t\in[0,T]\right)>1-\eta.
\]
The proof of this statement follows just as in
\cite[Lemma 5.12]{kaspi-ramanan}.

Next we show $\C$-tightness.
Define for $\zeta\in\D_{\calM}$,
$$ J(\zeta)=\int_0^\iy e^{-t}[J(\zeta, t)\wedge 1]dt , \quad \text{where}\, J(\zeta, t)=\sup_{s\leq t} d_\calM(\zeta_s, \zeta_{s-}).$$
By \cite[Theorem 3.10.2]{ethier-kurtz}, to show $\C$-tightness it suffices
to show that for any $\eps, \eta>0$
$$\liminf_{N\to\iy}\p(J(\bar{\nu}^N)\leq\eps)\geq 1-\eta.$$
However, this is obvious from Lemma~\ref{lem5.3}.
This completes the proof.
\qed
\end{proof}

\begin{lemma}\label{lem5.5}
Let Assumption~\ref{Assump1} hold. Then the collection of measure-valued
processes $\{\bar{\calQ}^N\}$ is $\C$-tight.
\end{lemma}

\begin{proof}
Note by \eqref{13} that $\bar\calQ^N$ is the image of $\bar\al^N$ and $\bar K^N+\bar R^N$
under $\Th$.
Since we have already proved $\C$-tightness of $\bar K^N$ and $\bar R^N$,
the result is an immediate consequence of the continuity of $\Th$
on $\C_{\calM}^{\up,0}\times\C_\R^\up$,
shown in \cite[Proposition 2.10]{atar-bis-kaspi-ram}.
\qed
\end{proof}

Next, introduce two measure-valued processes associated to $D^N$ and its compensator,
taking values in $\calM([0, H^s)\times R_+)$.
For $A$ a measurable subset of $[0, H^s)\times R_+$, let
\begin{align*}
\bar\calD^N_t(A) &=\bar{D}^N_t(\ind_A),
\\
\bar\calA^N_t(A) & =\bar{A}^N_t(\ind_A),
\end{align*}
where ${D}^N_t(\varphi)$ and ${A}^N_t(\varphi)$ are given by \eqref{9} and \eqref{30}, respectively.
For $\varphi\in\C_c([0, H^s)\times R_+)$, denote
$\bar\calD^N_t(\ph)=\bar\calD^N_t(\ph)$ and $\bar\calA^N_t(\ph)=\bar\calA^N_t(\ph)$.
Writing $\varphi=\varphi^+-\varphi^-$ and using Lemma~\ref{lem5.2} we have the sequences $\{\bar{D}^N(\varphi)\}, \{\bar{A}^N(\varphi)\}$ $\C$-tight, for every $\varphi\in\C_{b}([0, H^s)\times\R_+)$. Using \eqref{34} and the arguments of
\cite[Lemma~5.13]{kaspi-ramanan} one can show that the processes $\bar\calD^N, \bar\calA^N$ satisfy the compact containment condition
in the sense of Jakubowski. Thus we can establish Jakubowski's criteria for compactness of processes in $\D_{\calM}([0, H^s)\times \R_+)$
 for $\bar\calD^N, \bar\calA^N$
(see \cite[Lemma~5.13]{kaspi-ramanan}) and obtain the following result.

\begin{lemma}\label{lem5.6}
Let Assumption~\ref{Assump1} hold. Then the sequences $\{\bar\calD^N\}$ and $\{\bar\calA^N\}$
are tight in the space $\D_{\calM([0, H^s)\times\R_+)}$.
\end{lemma}

\section{Characterization of limits}\label{sec-char}
\beginsec

Finally, we prove Theorem \ref{Thm3}.
This section is devoted in the characterization of the subsequential limits of the scaled processes.
Given the tightness result from section~\ref{sec-tight} and the uniqueness of solutions to
the FME, it suffices to prove that any subsequential limit of the scaled processes
solves the FME. In other words, Theorem \ref{Thm3} is an immediate consequence of the following.

\begin{proposition}\label{Thm6.6}
Let the hypotheses of Theorem \ref{Thm3} hold.
If $(\calQ, \nu, K, R)$ is any
subsequential limit of $(\bar\calQ^N, \bar\nu^N, \bar{K}^N, \bar{R}^N)$
then it solves the FME \eqref{1000}--\eqref{fme9} with data $(\calQ_0, \nu_0,\al, X_0)$.
\end{proposition}

We define
$$\calY=\R_+\times(\D_{\R_+})^3\times(\D_{\calM})^2\times\calM_1([0, H^s))\times\D_{\calM_1([0, H^s))}\times(\D_{\calM([0, H^s)\times\R_+)})^2.$$
We equip $\calY$ with the product topology. Let
$$\bar{Y}^N=(\bar{X}^N_0, \bar{X}^N, \bar{K}^N, \bar{R}^N, \bar\al^N, \bar{\calQ}^N, \bar{\nu}_0, \bar\nu^N, \bar\calA^N, \bar\calD^N)\in\calY.$$
Applying Lemma~\ref{lem5.1} --\ref{lem5.6} we see that $\bar{Y}^N$ is a tight sequences and thus it has a convergent subsequence. Let $Y$ be one
of the subsequential limit of $\bar{Y}^N$. In fact, $Y$ would be of following form because of Lemma~\ref{lem5.1},
$$Y=(X_0, X, K, R, \al, \calQ, \nu_0, \nu, \calA, \calA), \quad \text{where}\; \; \calA\in \D_{\calM([0, H^s)\times\R_+)}.$$
Also because of our $\C$-tightness we have $X, K, R, \calQ, \nu$ continuous. Moving to the subsequence and applying Skorohod
representation theorem we can assume that $\bar{Y}^N\to Y$ a.s.\ as $N\to\iy$ on some probability space, say $(\Omega, \mathcal{F}, \p)$.

The following result follows from \cite[Proposition~5.17]{kaspi-ramanan}

\begin{lemma}\label{lem6.1}
Suppose Assumption~\ref{Assump1} and \ref{Assump2} holds. Then for every $\varphi\in\C_b([0, H^s)\times\R_+)$,
$$\calA_t(\varphi)= \int_0^t \mean{\varphi(\cdot, s) h(\cdot), \nu_s} \ ds, \quad t\in[0, \iy).$$
\end{lemma}
In view of Lemma~\ref{lem6.1} following relations hold a.s.:
for any $\varphi\in \C^{1}_c([0, H^s)\times [0, \infty))$,
\begin{align}\label{cfme5}
\lan \varphi, \nu_t\ran & = \lan \varphi(\cdot, 0), \nu_0\ran + \int_0^t \lan \varphi_x(\cdot, s)+\varphi_s(\cdot, s), \nu_s\ran ds
- \int_0^t\lan h\varphi(\cdot, s), \nu_s\ran ds + \int_0^t \varphi(0, s)dK_s,
\end{align}
where with $Q_t=\calQ_t(\R_+)$, we have
\begin{align}
Q_0 + E_t &= Q_t + K_t +R_t,\label{cfme1}
\\
X_t &= Q_t + \lan 1, \nu_t\ran, \label{cfme2}
\\
Q_t &=[X_t-1]^+, \label{cfme3}
\\
K_t &= \lan1, \nu_t\ran-\lan 1, \nu_0\ran + D_t = \lan1, \nu_t\ran-\lan 1, \nu_0\ran + \int_0^t \lan h, \nu_s\ran ds, \label{cfme4}
\end{align}
where \eqref{cfme5} follows from \eqref{12}, \eqref{cfme1}-\eqref{cfme4} follows from \eqref{3}-\eqref{7}. From \eqref{10}  and \eqref{6} we get
\begin{align}
\calQ_t[0, t) & =0,\label{cfme7}
\\
\int_0^\cdot [1-X_s]^+dR_s & =0.\label{cfme8}
\end{align}
Since $\calQ^N_t[a, b]\leq \al^N_t[a, b]$ we obtain from Assumption~\ref{Assump1} that $\calQ_t$ does not have any atoms for every $t\in[0, \iy)$
a.s. Therefore from Proposition~2.1 and Theorem~\ref{Thm2} we get that, a.s., for all $x\in [0, \iy)$,
\begin{equation}\label{cfme6}
\calQ_t[0, x]=\Gam_1[\al[0, x]-K-R](t), \quad t\geq 0.
\end{equation}
Thus to complete the proof of Proposition~\ref{Thm6.6}, it remains to show \eqref{fme9}.
Define
\[
\sigma_t=\inf\,\text{support}(\calQ_t)=\inf\{x\in\R_+\; :\; \calQ_t[0, x]>0\}.
\]
From \eqref{10} we have $\calQ_t[0, t)=0$ and therefore $\sigma_t\geq t$ for
all $t\in[0, \iy)$. We have the following result which is in analogy with \eqref{11}.

\begin{lemma}\label{lem6.2}
Let either Assumption~\ref{Assump3} or Assumption~\ref{Assump4} hold. Then we have a.s.,
$$
\int_0^\cdot \ind_{\{\sigma_s>s\}} dR_s=0.
$$
\end{lemma}

\begin{proof}
Fix $T>0$. We show that
$$\int_0^T \ind_{\{\sigma_s>s\}} dR_s=0.$$
Since $\{\sigma_s>s\}=\cup_{n\in\N}\{\sigma_s> s+\frac{1}{n}\}$, it is enough
to show that for any positive $\hat\del$ we have a.s.,
\begin{equation}\label{60}
\int_0^T \ind_{\{\sigma_s> s+\hat\del\}} dR_s=0.
\end{equation}
Notice that at this stage of the proof it has not yet been established
that the subsequential limit forms a solution to the FME, and therefore we cannot treat
it as deterministic.

The measurability of the set
$$A_0=\lt\{\int_0^T \ind_{\{\sigma_s> s+\hat\del\}} dR_s=0\rt\}$$
can be shown following the arguments in \cite[Lemma~5.9]{atar-bis-kaspi-ram}.
That is, from \eqref{cfme6} one can easily show that $t\mapsto \calQ_t[0, t+a]$
is right continuous for every $a\geq 0$ (see also \cite[Lemma~4.4]{atar-bis-kaspi-ram}).
Therefore the map $(\omega, t)\mapsto \calQ_t(\omega)[0, t+a]$ is optional and the set
$$
\{(t, \omega)\in [0, T)\times\Om \; :\;  \sigma_t(\omega)> t+\hat\delta\}
= \cup_{n=1}^\infty\{(t, \omega)\; :\; \calQ_t[0, t+\hat\delta+\frac{1}{n}]=0\},
$$
is also optional. This implies the measurability of
$$
\Gam=\{(t, \omega)\in[0, T)\times\Om\; :\; \sigma_t(\omega)> t+\hat\delta,
R_{t+\frac{1}{n}}>R_t\; \text{for all}\; n\},
$$
with respect to $\B([0, T))\times\calF_\infty$ where the filtration $\{\calF_t\}$
is obtained by augmenting the usual way the filtration
$\sigma\{\calQ_s, \nu_s, K_s, R_s\; :\; s\leq t\}$. Therefore
applying the Section Theorem for the measurable set $\Gam$, we can find
a $[0, T]\cup\{\infty\}$-valued random variable
$\tau$ such that $\p(A^c_0)=\p(A_1\cap A_2)$ for
$$
A_1=\{\tau<T\; :\; \sigma_\tau>\tau+\hat\delta\},
\quad \text{and} \quad A_2=\{R_{\tau+\eps}>R_\tau\; \text{for all }\, \eps>0\}.
$$
See \cite[Lemma~5.9]{atar-bis-kaspi-ram} for more details.

Therefore it is enough to show that on $A_1$,
that is, if $\sigma_\tau(\om)> \tau(\om)+\hat\del$
then there exists
$\eps=\eps(\omega)>0$ so that
$R_{\tau+\eps}(\omega)=R_\tau(\omega)$.
%Now we show that then there exists $\eps>0$ such that $R_{\tau+\eps}=R_\tau$.
This will prove $\p(A^c_0)=0$ and therefore
\eqref{60} holds a.s. We know that for any $a<b\in[0, T]$ we have
\begin{equation}\label{61}
\bar{R}^N_b-\bar{R}^N_a\leq \bar\calQ^N_a(a, b]+\bar\al^N_b(a, b] - \bar\al^N_a(a, b].
\end{equation}
The above follows from the fact that the total number of customers reneged in the interval $(a, b]$ must be smaller than the number of customers in
the queue at time $a$ with absolute deadlines in $(a, b]$ together with the total number of customers arrived in time interval $(a, b]$ with absolute deadline in
$(a, b]$. Let
$\eps_J=J^{-1}\eps$ for some $J\in\N$ and let $I_k, k=1,\ldots, J$ denote the partition
$$I_k=(t_{k-1}, t_k], \quad t_k=\tau+k\eps_J,$$
of $(\tau, \tau+\eps]$. From  \eqref{61}
\begin{align}
\bar{R}^N_{\tau+\eps} - \bar{R}^N_{\tau} &= \sum_{k=1}^J(\bar{R}^N_{t_k}-\bar{R}^N_{t_{k-1}})\nonumber
\\
&\le\sum_{k=1}^J\Big(\bar\al^N_{t_k}(I_k)-\bar\al^N_{t_{k-1}}(I_k)+\bar\calQ^N_{t_{k-1}}(I_k)\Big).\label{62}
\end{align}
Since
$$\sup_{s\in[0, T]}d_\calM(\bar\calQ^N_s,\calQ_s)\to 0, \quad \text{as}\quad N\to\iy,\; \text{a.s.},$$
and $\calQ$ does not have atoms we get
\begin{equation}\label{63}
\sup_{s\in[0, T]}\sup_{x\in\R_+}|\bar\calQ^N_s(x, \iy)-\calQ_s(x, \iy)|\to 0, \quad \text{as}\; N\to\iy, \; \text{a.s.}
\end{equation}
On $A_1$ we have $\calQ_\tau[0, \tau+\hat\del]=0$. Now suppose that Assumption~\ref{Assump3} holds. Since $(\nu, K)$
satisfies \eqref{cfme5} for every sample path, it also satisfies the conclusion of Lemma~\ref{lem3.2}.
Thus we can find $\del$ and a compact set $\bC\subset[0, H^s)\setminus\calZ$ so that
\begin{align*}
\sup_{s\in[0, T]}\nu_s(\bC^c) &<\frac{1}{4},\quad \sup_{s\in[0, T]}\la_s\pi[0, \del]<\frac{3}{4}\inf_{x\in\bC}h(x),
\\
G(\del)<\frac{1}{10}\, .
\end{align*}
 Take any $\eps\in(0, \hat\del\wedge\del)$.
Then one can follow the proof of Lemma~\ref{lem5}
to conclude that $\calQ_s[0, \tau+\eps]=0$ for all $s\in[\tau, \tau+\eps]$ almost surely on $A_1$. Now we deduce similar conclusion under Assumption~\ref{Assump4}.
Let Assumption \ref{Assump4} hold. Since $\sigma_\tau>\tau+\hat\del$ on $A_1$ we also have
$\calQ_\tau[0, \tau+\hat\del]=0$.
 Denote $\hat\al_s=\al_s-\al_\tau$. Similarly, define
$\hat{K}_s=K_s-K_\tau, \; \hat{R}_s=R_s-R_\tau$. Then from \eqref{cfme6} we obtain for $s\in[\tau, T]$ that for any $x\leq \tau+\hat\del$
\begin{align*}
\calQ_s[0, x] &=\calQ_\tau[0, x] +\hat\al_s[0, x]-\hat{K}_s-\hat{R}_s +\sup_{u\leq s}(\hat{K}_s+\hat{R}_s-
\hat\al_s[0, x]
-\calQ_\tau[0, x])^+
\\
&=\hat\al_s[0, x]-\hat{K}_s-\hat{R}_s +\sup_{u\leq s}(\hat{K}_s+\hat{R}_s-\hat\al_s[0, x])^+.
\end{align*}
Since $\pi[0, \kappa_1]=0$, if we choose $\eps=\frac{\kappa_1}{2}\wedge\hat\del$ we have for $s\leq \tau+\hat\del$ that
$$\hat{\al}_s[0, \tau+\eps]=\int_\tau^s\la_u\ind_{\{\tau+\eps\geq u\}}\pi[0, \tau+\eps-u] du =0.$$
Thus combining the above two displays we have $\calQ_s[0, \tau+\eps]=0$ for $s\in[\tau, \tau+\eps]$. Thus using \eqref{63}
we get
\begin{equation}\label{64}
\sup_{s\in[\tau, \tau+\eps]}\bar\calQ^N_s[0, \tau+\eps]\to 0, \quad \text{as}\;\; N\to\iy, \; \text{a.s. on}\, A_1.
\end{equation}
It should be noted that our choice of $\varepsilon$ is non-random on $A_1$. Now we consider the summation
\begin{align*}
S^N_J= \sum_{k=1}^J\Big(\bar\al^N_{t_k}(I_k)-\bar\al^N_{t_{k-1}}(I_k)\Big).
\end{align*}
For fixed $J$ we see that as $N\to\iy$ we get $S^N_J\to S_J$ a.s., where
\begin{align*}
S_J &= \sum_{k=1}^J\Big(\al_{t_k}(I_k)-\al_{t_{k-1}}(I_k)\Big),
\\
&=\sum_{k=1}^J \int_{t_{k-1}}^{t_k}\ind_{\{t_k\geq u\}}\la_s\pi[0, t_k-u] du
\\
&\leq T\sup_{s\in[0, T]}\la_s\, \pi[0, \eps_J],
\end{align*}
where the first equality is due to the fact that $\al$ does not have atoms. Therefore letting $N\to\iy$ in \eqref{62}
and using \eqref{64}, we have on $A_1$
$$R_{\tau+\eps}-R_{\tau}\leq T\sup_{s\in[0, T]}\la_s\, \pi[0, \eps_J],$$
where $J$ is arbitrary. Since $\lim_{x\to 0}\pi[0, x]=0$ we get from above that $R_{\tau+\eps}-R_{\tau}=0$
almost surely on $A_1$. This completes the proof.
\qed

\end{proof}

\skp

\noi{\bf Acknowledgment.}
The authors are indebted to the anonymous referees for their careful review and suggestions.
AB acknowledges the hospitality of the Department of Electrical Engineering in Technion
while he was visiting at the early stages of this work.
The research of RA was supported in part by grant 1184/16 of the Israel Science Foundation.
The research of AB was supported in part by an INSPIRE faculty fellowship.
The research of HK was supported in part by grant 764/13 of the Israel Science Foundation.

\footnotesize

\bibliographystyle{is-abbrv}
%\bibliographystyle{is-alpha}

%\bibliography{refs}

\end{document}